\documentclass[american]{amsart}

\usepackage[utf8]{inputenc}
\usepackage{graphicx}
\usepackage{pb-diagram}
\usepackage[backend=bibtex,style=alphabetic,isbn=false]{biblatex}
\usepackage{hyperref}
\usepackage{color}
\usepackage{xypic}

\usepackage{ifthen}
\DeclareMathOperator{\Aut}{Aut}

\DeclareMathOperator{\ev}{ev}
\DeclareMathOperator{\GL}{GL}

\DeclareMathOperator{\ImaginaryPart}{Im}

\DeclareMathOperator{\SO}{SO}

\DeclareMathOperator{\SU}{\mathrm{SU}}

\newcommand{\conSum}{{\mathbin{\,\#\,}}}
\newcommand{\restricted}[2]{{\left.{#1}\right|_{#2}}}

\newcommand{\p}{\partial}
\newcommand{\lie}[1]{{\mathcal{L}_{#1}}}
\newcommand{\abs}[1]{{\left\lvert #1\right\rvert}}
\newcommand{\norm}[1]{{\lVert #1\rVert}}

\newcommand{\attachingSphere}[1]{{S_\mathrm{at}^{#1}}}
\newcommand{\beltSphere}[1]{{S_\mathrm{belt}^{#1}}}
\newcommand{\Mreg}{{M_\mathrm{reg}}}

\newcommand{\Wmodel}{{W_\mathrm{model}}}

\renewcommand{\epsilon}{\varepsilon}
\makeatletter
\usepackage{xspace}

\makeatother
\newcommand{\LOB}{\textsf{Lob}\xspace}

\newcommand{\0}{{\mathbf 0}}
\newcommand{\bbf}{{\mathbf b}}
\newcommand{\bB}{{\mathcal B}}
\newcommand{\CC}{{\mathbb C}}
\newcommand{\dD}{{\mathcal D}}
\newcommand{\DD}{{\mathbb D}}

\newcommand{\nN}{{\mathcal N}}

\newcommand{\RR}{{\mathbb R}}
\renewcommand{\SS}{{\mathbb S}}

\newcommand{\uU}{{\mathcal U}}

\newcommand{\x}{{\mathbf x}}
\newcommand{\y}{{\mathbf y}}
\newcommand{\z}{{\mathbf z}}
\newcommand{\ZZ}{{\mathbb Z}}

\newcommand{\defin}[1]{\textbf{#1}}

\newcommand{\mM}{{\mathcal M}}
\newcommand{\mMsmooth}{\mM_{{\operatorname{smooth}}}}
\newcommand{\mMbubble}{\mM_{{\operatorname{bubble}}}}
\newcommand{\mMint}{\mM_{{\operatorname{int}}}}

\theoremstyle{plain}

\newcounter{maintheorem}

\newtheorem{theorem}{Theorem}[section]
\newtheorem{lemma}[theorem]{Lemma}

\newtheorem{corollary}[theorem]{Corollary}

\newtheorem{question}[theorem]{Question}

\newtheorem{proposition}[theorem]{Proposition}

\theoremstyle{remark}
\newtheorem{remark}[theorem]{Remark}
\newtheorem{example}[theorem]{Example}
\newtheorem{assumptions}[theorem]{Assumptions}

\theoremstyle{definition}

\newtheorem{definitionNumbered}[theorem]{Definition}

\numberwithin{equation}{section}

\begin{document}
\title[Subcritical surgery and symplectic fillings]{Subcritical
  contact surgeries and the topology of symplectic fillings}

\author{Paolo Ghiggini}
\address[P.\ Ghiggini]{
  Laboratoire de Mathématiques Jean Leray \\
  BP 92208 \\
  2, Rue de la Houssinière \\
  F-44322 Nantes Cedex 03 \\
  FRANCE}
\email{paolo.ghiggini@univ-nantes.fr}

\author{Klaus Niederkrüger}
\address[K.\ Niederkrüger]{
  Alfréd Rényi Institute of Mathematics \\
  Hungarian Academy of Sciences \\
  POB 127 \\
  H-1364 Budapest \\
  HUNGARY}
\address[K.\ Niederkrüger]{
  Institut de mathématiques de Toulouse\\
  Université Paul Sabatier -- Toulouse III\\
  118 route de Narbonne\\
  F-31062 Toulouse Cedex 9\\
  FRANCE}
\email{niederkr@math.univ-toulouse.fr}

\author{Chris Wendl}

\address[C.\ Wendl]{
  Department of Mathematics \\
  University College London \\
  Gower Street \\
  London WC1E 6BT \\
  UNITED KINGDOM }
\email{c.wendl@ucl.ac.uk}

\begin{abstract}
  By a result of Eliashberg, every symplectic filling of a
  three-dimensional contact connected sum is obtained by performing a
  boundary connected sum on another symplectic filling.
  We prove a partial generalization of this result for subcritical
  contact surgeries in higher dimensions:
  given any contact manifold that arises from another contact manifold
  by subcritical surgery, its belt sphere is null-bordant in the
  oriented bordism group~$\Omega_*^{SO}(W)$ of any symplectically
  aspherical filling~$W$, and in dimension five, it will even be
  nullhomotopic.
  More generally, if the filling is not aspherical but is
  semipositive, then the belt sphere will be trivial in $H_*(W)$.
  Using the same methods, we show that the contact connected sum
  decomposition for tight contact structures in dimension three does
  not extend to higher dimensions:
  in particular, we exhibit connected sums of manifolds of dimension
  at least five with Stein fillable contact structures that do not
  arise as contact connected sums.
  The proofs are based on holomorphic disk-filling techniques, with
  families of Legendrian open books (so-called ``\LOB{}s'') as
  boundary conditions.
\end{abstract}

\maketitle

\section{Introduction}

\subsection{The main result and corollaries}
\label{sec:main}
The idea of constructing contact manifolds as boundaries of symplectic
$2n$-manifolds by attaching handles of index at most $n$ goes back to
Eliashberg \cite{Eliashberg_Stein} and Weinstein
\cite{WeinsteinHandlebodies}.
In this context, a special role is played by \emph{subcritical}
handles, i.e.~handles with index strictly less than~$n$.
One well-known result on this topic concerns subcritical Stein
fillings, which are known to be \emph{flexible} in the sense that
their symplectic geometry is determined by homotopy theory, see
\cite{CieliebakEliashberg}.
There are also known restrictions on the topological types of
subcritical fillings, e.g.~by results of M.-L.~Yau
\cite{YauSubcritical} and Oancea-Viterbo
\cite[Prop.~5.7]{OanceaViterbo}, the homology of a subcritical filling
with vanishing first Chern class is uniquely determined by its contact
boundary; this result can be viewed as a partial generalization of the
Eliashberg-Floer-McDuff theorem \cite{McDuff_contactType} classifying
symplectically aspherical fillings of standard contact spheres up to
diffeomorphism.
In dimension three, there is a much stronger result due to Eliashberg
\cite{Eliashberg_filling,CieliebakEliashberg}: in this case every
subcritical surgery is a connected sum, and the result states
that if $(M',\xi')$ is a closed contact $3$-manifold obtained from
another (possibly disconnected) contact manifold $(M,\xi)$ by a
connected sum, then every symplectic filling of $(M',\xi')$ is
obtained by attaching a Weinstein $1$-handle to a symplectic filling
of $(M,\xi)$.
This implies that symplectic fillings of subcritically fillable
contact $3$-manifolds are unique up to symplectic deformation equivalence 
and blowup---in particular, their Stein fillings are unique up to
symplectomorphism.
The present paper was motivated by the goal of generalizing
Eliashberg's connected sum result to higher dimensions.
The natural question in this setting is the following:

\begin{question}\label{question:main}
  Given a closed contact manifold $(M',\xi')$ that is obtained from
  another contact manifold $(M,\xi)$ by subcritical contact surgery,
  is every (symplectically aspherical) filling of $(M',\xi')$
  obtained by attaching a subcritical Weinstein handle to a symplectic
  filling of $(M,\xi)$?
\end{question}

Classifying fillings up to symplectomorphism as suggested in this
question would be far too ambitious in higher dimensions, e.g.~the
strongest result known so far, the Eliashberg-Floer-McDuff theorem, is
essentially a classification of fillings up to \emph{homotopy type}
(the $h$-cobordism theorem then improves it to a classification up to
diffeomorphism).
Our objective in this paper will therefore be to understand the main
\emph{homotopy-theoretic} obstruction to an affirmative answer to
Question~\ref{question:main}.

To state the main result, let us first recall some basic notions.
Given an oriented $(2n-1)$-dimensional manifold $M$, a (positive,
co-oriented) \defin{contact structure} on $M$ is a hyperplane
distribution of the form $\xi = \ker \alpha$, where the \defin{contact
  form} $\alpha$ is a smooth $1$-form satisfying $\alpha \wedge
(d\alpha)^{n-1} > 0$, and the co-orientation of~$\xi$ is determined by
$\alpha > 0$.
In this paper, contact structures will \emph{always} be assumed to be
both positive (with respect to a given orientation on~$M$) and
co-oriented, and all contact forms will be assumed compatible with the
co-orientation.
A compact symplectic $2n$-manifold $(W,\omega)$ with oriented boundary
$M = \p W$ carrying a contact structure $\xi$ is called a
\defin{strong symplectic filling} of $(M,\xi)$ if $\xi$ admits a
contact form~$\lambda$ that extends to a primitive of $\omega$ on a
neighborhood of~$\p W$.
It is equivalent to say that the boundary is \defin{symplectically
  convex}, as the vector field $\omega$-dual to $\lambda$ is then a
Liouville vector field pointing transversely outward at~$\p W$.
More generally, we say that $(W,\omega)$ is a \defin{weak symplectic
  filling} of $(M,\xi)$ if $\xi$ is the bundle of complex tangencies
for some $\omega$-tame almost complex structure near $\p W$ that makes
the boundary pseudoconvex (see \cite{WeafFillabilityHigherDimension}).
Recall that if $(M,\xi)$ contains a $(k-1)$-dimensional isotropic
sphere $\attachingSphere{k-1}$ with trivial normal bundle, then one
can perform a \defin{contact surgery of index~$k$} on $(M,\xi)$ by
attaching to $(-\epsilon,0] \times M$ a handle of the form $\DD^k
\times \DD^{2n-k}$ along a neighborhood of~$\attachingSphere{k-1}$.
The new contact manifold $(M',\xi')$ then contains the
$(2n-k-1)$-dimensional coisotropic sphere $\beltSphere{2n-k-1} = \{0\}
\times \p \DD^{2n-k}$, which we call the \defin{belt sphere} of the
surgery.
This surgery operation was first introduced by Weinstein
\cite{WeinsteinHandlebodies}, and we will give a more precise
description of it in Section~\ref{sec: description handle and
  surgery}.
A Weinstein handle yields a symplectic cobordism that can be attached
to any weak filling $(W,\omega)$ of $(M,\xi)$ for which $\omega$ is
exact along $\attachingSphere{k-1}$; the result is a weak filling of
$(M',\xi')$ in which the belt sphere is necessarily contractible
(Figure~\ref{fig:trivial}).
Our main result is the following.

\begin{theorem}\label{thm: main theorem}
  Suppose $(M',\xi')$ is a closed contact manifold of dimension~$2n-1
  \ge 3$ that has been obtained from a manifold $(M,\xi)$ by a contact
  surgery of index $k \le n-1$, with belt sphere $\beltSphere{2n-k-1}
  \subset M'$.
  Assume $(W',\omega')$ is a weak symplectic filling of $(M',\xi')$.
  \begin{itemize}
  \item[(a)] If ($W',\omega')$ is semipositive, then the belt sphere
    represents the trivial homology class in $H_{2n-k-1}(W'; \ZZ)$.
  \item[(b)] If $(W',\omega')$ is symplectically aspherical, then the
    belt sphere represents the trivial element in the oriented
      bordism group $\Omega_{2n-k-1}^{SO}(W')$.
    If additionally either (1)~$M'$ is $5$-dimensional, or (2)~$M'$ is
    $7$-dimensional and $k=3$, then $\beltSphere{2n-k-1}$ is
    contractible in $W'$, that is, it represents the trivial class in
    $\pi_{2n-k-1}(W')$.
  \end{itemize}
\end{theorem}

\begin{figure}[htbp]
  \centering
  \includegraphics[height=2.5cm,keepaspectratio]{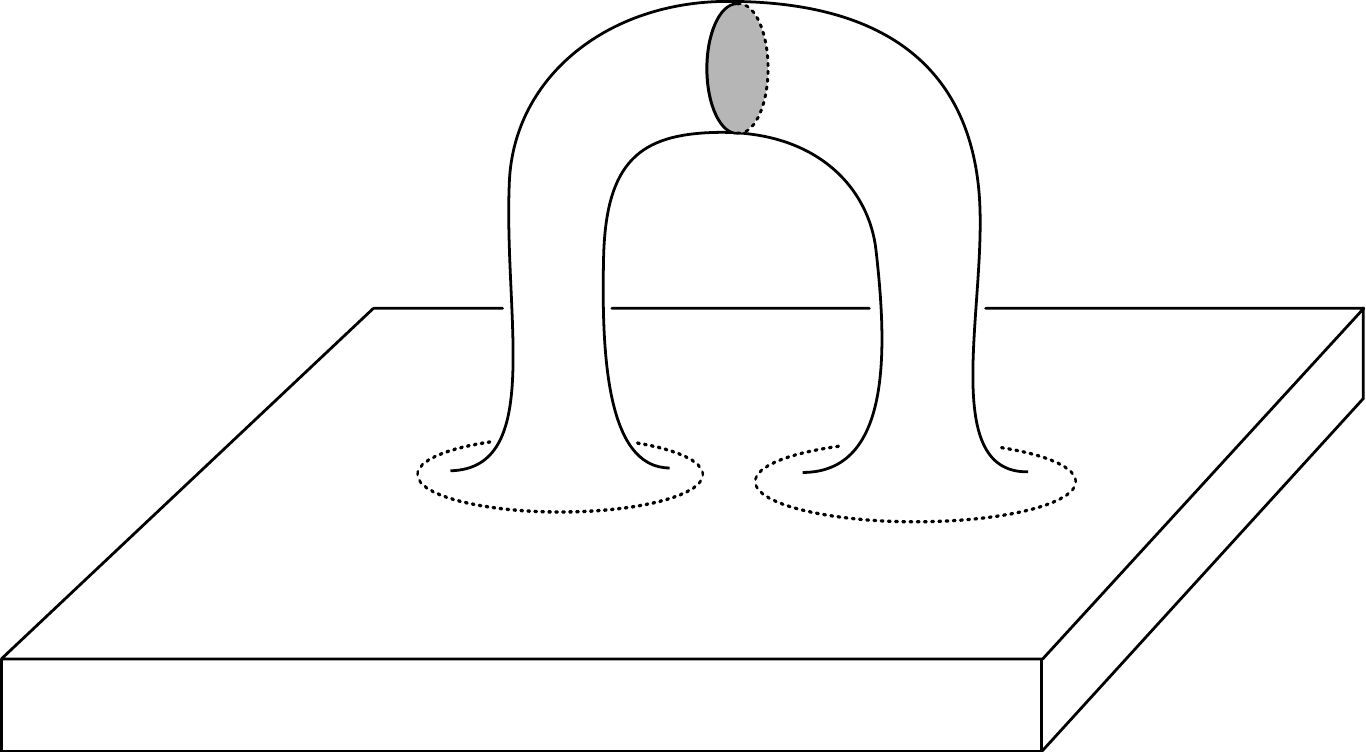}
  \caption{\label{fig:trivial} The belt sphere bounds an embedded
    disk inside the handle.}
\end{figure}

We now state two related results that follow via the same techniques.
Recall that in dimension three, convex surface theory gives rise to a
contact prime decomposition theorem, implying e.g.~that every tight
contact structure on a closed $3$-manifold of the form $M_0 \# M_1$
arises as a contact connected sum of tight contact structures on $M_0$
and $M_1$ (see \cite{ColinConnectedSum}, or \cite[\S 4.12]{GeigesBook}
for more details).
Some evidence against a higher-dimensional generalization of this
result appeared in the recent work of Bowden, Crowley and Stipsicz
\cite{BowdenCrowleyStipsicz2}, providing also a negative answer to a
topological version of Question~\ref{question:main}:
namely, there exist pairs of closed oriented manifolds $M_0, M_1$ such
that $M_0 \conSum M_1$ admits a Stein fillable contact structure but
$M_0$ and $M_1$ do not.
This did not imply an actual answer to Question~\ref{question:main},
however, as it was unclear whether the contact structures on $M_0
\conSum M_1$ in the examples of \cite{BowdenCrowleyStipsicz2} could
actually be \emph{contact} connected sums, i.e.~whether they arise
from contact structures $\xi_0$ on $M_0$ and $\xi_1$ on $M_1$ by
performing index~$1$ contact surgery.
The following result gives a negative answer to the latter question,
and shows that there is no hope of extending the contact prime
decomposition theorem to higher dimensions.
The theorem applies in particular whenever $M$ is an almost contact
$\SS^{n-1}$-bundle over~$\SS^n$ that is not a homotopy sphere, so for
instance $M$ could be $\SS^{n-1} \times \SS^n$ or---as in
\cite{BowdenCrowleyStipsicz2}---the unit cotangent bundle of a sphere.

\begin{theorem}\label{thm: almost contact connSum not genuine}
  Suppose $M$ is a closed oriented manifold of dimension $2n-1 \ge 5$
  that is not a homotopy sphere but admits an almost contact structure
  $\Xi$ and a Morse function with unique local maxima and minima and
  otherwise critical points of index $n-1$ and $n$ only.
  Then $M\conSum (-M)$ admits a Stein fillable contact structure that
  is homotopic to the almost contact structure $\Xi \conSum
  \overline{\Xi}$ but is not isotopic to $\xi_1 \conSum \xi_2$ for any
  contact structures $\xi_1$ and $\xi_2$ on $M$ and $-M$ respectively.
  \footnote{Given an oriented manifold~$M$ with almost contact
    structure~$\Xi$, we denote by $-M$ the same manifold with reversed
    orientation, and let $\overline{\Xi}$ denote the almost contact
    structure on $-M$ obtained by inverting the co-orientation
    of~$\Xi$.}
\end{theorem}

Note that the contact structures in the above statement are
necessarily tight in the sense of Borman-Eliashberg-Murphy
\cite{BormanEliashbergMurphy_wow}; this follows from Stein
fillability, using \cite{NiederkrugerPlastikstufe} and the observation
in \cite{BormanEliashbergMurphy_wow} that any overtwisted contact
structure is also PS-overtwisted.
The holomorphic disk techniques developed in this article can also be
used as in the work of Hofer \cite{HoferWeinstein} to prove the
Weinstein conjecture for a wide class of contact manifolds obtained by
subcritical surgery.
The following theorem, proved in Section~\ref{sec:Weinstein}, is
related to the well-known result that every subcritically Stein
fillable contact form admits a contractible Reeb orbit.
(A similar result specifically for index~$1$ surgeries appeared
recently in \cite{GeigesZehmisch}.)

\begin{theorem}\label{thm:Weinstein}
  Assume $(M',\xi')$ is the result of performing a contact surgery of
  index $k \le n-1$ on a closed contact manifold $(M,\xi)$ of
  dimension $2n-1 \ge 3$, with belt sphere $\beltSphere{2n-k-1}
  \subset M'$, and suppose at least one of the following conditions
  holds:
  \begin{enumerate}
  \item $[\beltSphere{2n-k-1}] \ne 0$ in $\Omega_{2n-k-1}^{SO}(M')$;
  \item $[\beltSphere{2n-k-1}] \ne 0$ in $\pi_{2n-k-1}(M')$ and
    either $\dim M' = 5$ or $\dim M' = 7$ with $k=3$;
  \item $\dim M' = 5$ and $(M',\xi')$ is a contact connected sum
    $(M_0,\xi_0) \conSum (M_1,\xi_1)$ with the following two
    properties:
    \begin{enumerate}
    \item Neither $M_0$ nor $M_1$ is homeomorphic to a sphere;
    \item If $M_0$ and $M_1$ are both rational homology spheres, then
      either both are not simply connected or at least one of them has
      infinite fundamental group.
    \end{enumerate}
  \end{enumerate}
  Then every contact form for $\xi$ admits a contractible Reeb orbit.
\end{theorem}

Before discussing the proofs, some further remarks about the main
theorem are in order.

\begin{remark}
  We do not know whether the dimensional restriction for the
  contractibility result in part~(b) of Theorem~\ref{thm: main
    theorem} is essential, but given the wide range of known contact
  geometric phenomena that can happen \emph{only} in sufficiently high
  dimensions, we consider it plausible that the contractibility
  statement could be false without some restriction of this type (thus
  implying a definitively negative answer to
  Question~\ref{question:main} in general).
  It is apparent in any case that our method will not work in all
  dimensions, as the improvement from ``null-bordant'' to
  ``nullhomotopic'' involves subtle topological difficulties that
  increase with the dimension; see the beginning of Section~\ref{sec:
    surgery on the moduli space} for more discussion of this.
\end{remark}

\begin{remark}
  It is clear that nothing like Theorem~\ref{thm: main theorem} can be
  true for \emph{critical surgeries} in general, i.e.~the case $k=n$.
  There are obvious counterexamples already in dimension three, as any
  Legendrian knot $L \subset (M',\xi')$ can be viewed as the belt
  sphere arising from a critical contact surgery---take $(M,\xi)$ in
  this case to be the result of a Legendrian $(+1)$-surgery along~$L$.
  It is certainly not true in general that arbitrary Legendrian knots
  are nullhomologous in every filling of $(M,\xi)$!
\end{remark}

\begin{remark}\label{remark:aspherical}
  The semipositivity assumption in part~(a) of Theorem~\ref{thm: main
    theorem} is there for technical reasons and could presumably be
  lifted using more advanced technology (e.g.~polyfolds, see
  \cite{HoferWZ_GW}).
  In contrast, symplectic asphericity in part~(b) is a geometrically
  meaningful condition that, while not needed for Eliashberg's
  three-dimensional version of this result, cannot generally be
  removed in higher dimensions; see Example~\ref{ex:aspherical} below.
  The answer to Question~\ref{question:main} thus becomes negative
  without this assumption.
\end{remark}

\begin{example}\label{ex:aspherical}
  The blowup of the total space of the rank~$2$ holomorphic vector
  bundle ${\mathcal O}(-2) \oplus {\mathcal O}$ over $\CC P^1$ at the
  zero section can be viewed as a (not symplectically aspherical) weak
  filling of a subcritically Stein fillable contact manifold
  $(M',\xi')$ containing a belt sphere that is homotopically
  nontrivial in the filling.
  This is a special case of the following construction, which gives
  examples with subcritical handles of any even index $k=2m \ge 2$ in
  any dimension $2n \ge 2k+2 \ge 6$.
  Choose integers $m , \ell \ge 1$ and set $n = 2m + \ell$, and
  suppose $(\Sigma,\sigma)$ is a $2m$-dimensional closed symplectic
  manifold.
  Then consider the $2n$-dimensional Weinstein manifold $T^*\Sigma
  \times \CC^\ell$, i.e.~the $\ell$-fold stabilization of $T^*\Sigma$
  with its standard Weinstein structure, and denote its ideal contact
  boundary by $(M',\xi')$.
  Any Morse function on $\Sigma$ gives rise to a Weinstein handle
  decomposition of $T^*\Sigma \times \CC^\ell$, such that the
  function's maximum $q \in \Sigma$ corresponds to an
  $(n-\ell)$-handle whose belt sphere $\beltSphere{n+\ell-1}$ is
  isotopic to the unit sphere in $T_q^*\Sigma \times \CC^\ell$.
  Let $\Sigma \subset T^*\Sigma$ denote the zero section, so $\Sigma
  \times \{0\}$ is an isotropic submanifold in $T^*\Sigma \times
  \CC^\ell$, and denote by $\pi \colon T^*\Sigma \times \CC^\ell \to
  \Sigma \times \{0\}$ the obvious projection.
  Then for any $\epsilon > 0$ sufficiently small, adding $\epsilon
  \pi^*\sigma$ to the natural exact symplectic form on $T^*\Sigma
  \times \CC^\ell$ gives a weak filling of $(M',\xi')$ with $\Sigma
  \times \{0\}$ as a symplectic submanifold.
  We can then blow up along this submanifold, as explained in
  \cite[Section~7.1]{McDuffSalamonIntro}.
  This produces a new weak filling $(W',\omega')$ of $(M',\xi')$, in
  which the belt sphere $\beltSphere{n+\ell-1} \subset M'$ is
  nullhomologous but homotopically nontrivial:
  indeed, every fiber $T_q^*\Sigma \times \CC^\ell$ has now been
  replaced by its blowup at the point $(0,0)$, which can be viewed as
  the tautological line bundle over $\CC P^{m+\ell-1}$, so the bundle
  projection sends $\beltSphere{n+\ell-1}$ to a generator of
  $\pi_{2(m+\ell)-1}(\CC P^{m+\ell-1}) \cong \ZZ$.
  The special case with $\Sigma = \SS^2$ and $\ell=1$ gives the
  construction described at the beginning of this example, because the
  total space of ${\mathcal O}(-2)$ is a deformation of $T^*\SS^2$.
\end{example}

The following represents another easy application of Theorem~\ref{thm:
  main theorem}.

\begin{example}
  Suppose $(M_1,\xi_1)$ is a contact $5$-manifold obtained by a
  subcritical surgery of index~$2$ on a sphere $(\SS^5,\xi)$, where
  $\xi$ is any contact structure.
  Then $M_1$ is diffeomorphic to either $\SS^2 \times \SS^3$ or
  $\SS^2\tilde \times \SS^3$, i.e.~the trivial or nontrivial
  $3$-sphere bundle over the $2$-sphere.
  Indeed, closed loops in $\SS^5$ are automatically unknotted, and the
  possible framings of the surgery are classified by the elements of
  $\pi_1\bigl(\SO(3)\bigr) \cong \ZZ_2$.
  If $(W,\omega)$ is any symplectically aspherical weak filling of
  $(M_1,\xi_1)$, then by Theorem~\ref{thm: main theorem}, the fiber
  $\{p\}\times \SS^3$ is a contractible $3$-sphere in $W$.
  Now take $(M_2,\xi_2)$ to be the unit cotangent bundle of the
  $3$-sphere or, more generally, any contact manifold supported by a
  contact open book with page $T^*\SS^2$ and monodromy isotopic to a
  $2k$-fold product of Dehn twists for some integer $k \ge 1$.
  Then $M_2$ will be diffeomorphic to $\SS^2\times \SS^3$, but
  $(M_2,\xi_2)$ admits a Stein filling that contracts to a bouquet of
  $2k-1$ three-dimensional spheres (see
  e.g.~\cite{vanKoertNiederkruger2005}).
  We conclude that whenever $(M_1,\xi_1)$ admits a symplectically
  aspherical weak filling, it is not contactomorphic to $(M_2,\xi_2)$.
  This implies for instance that the contact structures induced on the
  ideal contact boundaries of $T^*\SS^3$ and $T^*\SS^2 \times \CC$
  (cf.~Remark~\ref{remark:aspherical}) are not isomorphic.
  There are presumably other ways to distinguish $\xi_1$ and $\xi_2$
  in many cases, e.g.~using Symplectic Homology, but the technique
  described above is much more topological.
\end{example}

\subsection{Idea of the proof}
\label{sec:proof}

Our proof of Theorem~\ref{thm: main theorem} is based on a
higher-dimensional analogue of the disk-filling methods underlying
Eliashberg's result in dimension three \cite{Eliashberg_filling}.
Such methods work whenever one can find a suitable submanifold to
serve as a boundary condition for holomorphic disks, and the natural
object to consider in this case is known as a \emph{Legendrian open
  book} or ``\LOB{}''.
Let us recall the definition, which is due originally to the second
author and Rechtman \cite{NiederkrugerRechtman}.

\begin{definitionNumbered}\label{defn:Lob}
  A \LOB{} in a $(2n-1)$-dimensional contact manifold $(M,\xi)$ is a
  closed $n$-dimensional submanifold $L \subset M$ equipped with an
  open book decomposition $\pi \colon L \setminus B \to \SS^1$ whose
  binding $B \subset L$ is an $(n-2)$-dimensional isotropic
  submanifold of $(M,\xi)$, and whose pages $\pi^{-1}(*)$ are each
  Legendrian submanifolds of $(M,\xi)$.
\end{definitionNumbered}

The simplest interesting example of a \LOB{} occurs at the center of
the ``neck'' in any $3$-dimensional contact connected sum:
here we find a $2$-sphere $S \subset M$ on which the characteristic
foliation $\xi \cap TS$ traces an $\SS^1$-family of longitudes
connecting the north and south poles, so we can regard the longitudes
as pages of an open book with the poles as binding.
Such spheres were used as totally real boundary conditions for
holomorphic disks in \cite{BedfordGaveau, Gromov_HolCurves,
  Eliashberg_filling}, and similarly in Hofer's proof of the Weinstein
conjecture \cite{HoferWeinstein} for contact $3$-manifolds $(M,\xi)$
with $\pi_2(M) \ne 0$.
In higher dimensions, a \LOB{} $L \subset (M,\xi)$ with binding $B
\subset L$ similarly defines a totally real submanifold $\{0\} \times
(L \setminus B)$ in the symplectization $\RR \times M$ of $(M,\xi)$,
and thus serves as a natural boundary condition for pseudoholomorphic
disks.
\begin{figure}[htbp]
  \centering
  \includegraphics[height=3cm,keepaspectratio]{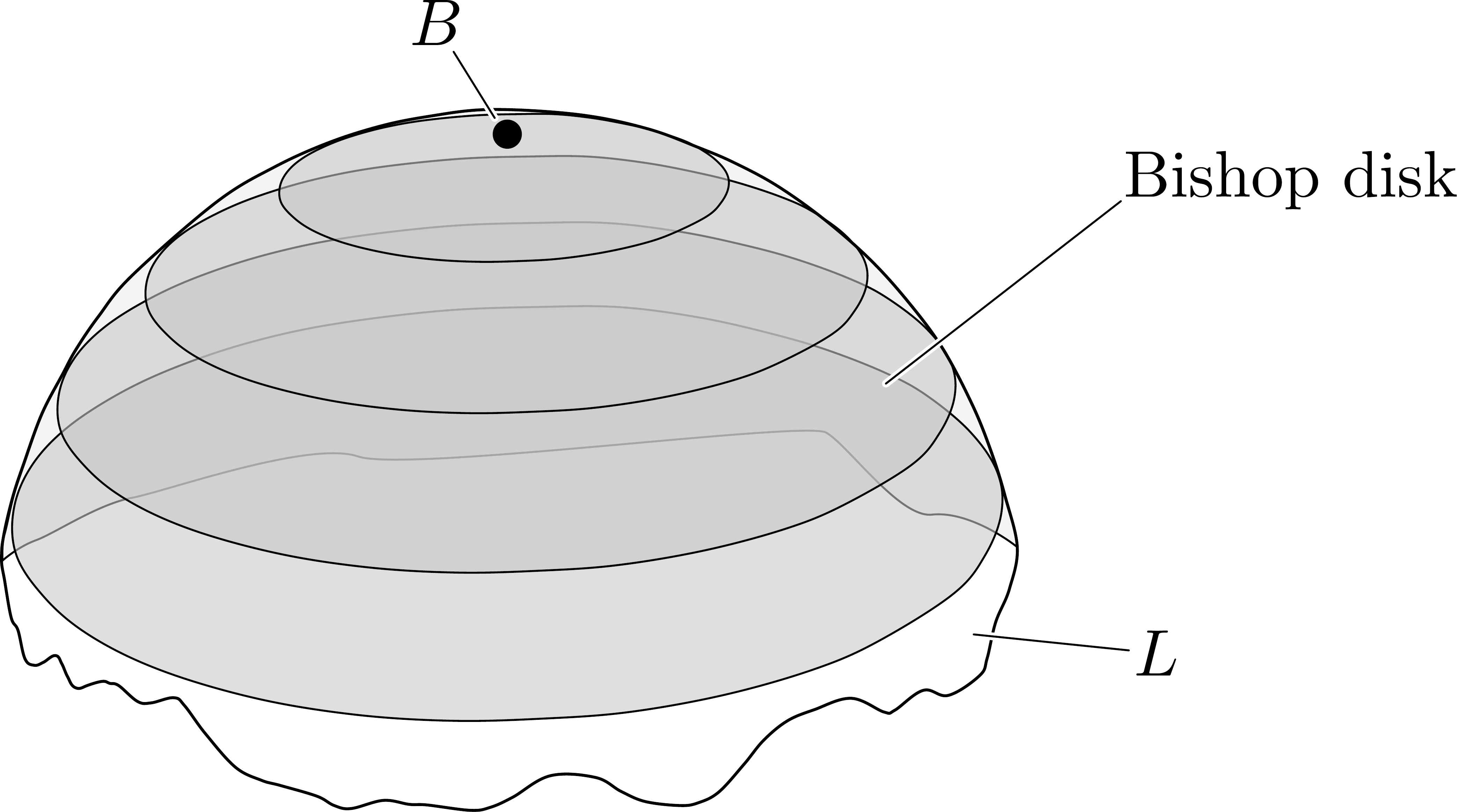}
  \caption{A schematic picture of the Bishop family around the binding
    of a \LOB~$L$.}\label{fig: schematic Bishop family}
\end{figure}
Moreover, for a suitably ``standard'' choice of almost complex
structure near the binding, a \LOB{} always gives rise to a canonical
family of holomorphic disks near $\{0\} \times B$ whose boundaries
foliate a neighborhood of $B$ in $L \setminus B$ (see Figure~\ref{fig:
  schematic Bishop family}).
This is the so-called \defin{Bishop family} of holomorphic disks, and
it has the useful property that no other holomorphic curve can enter
the region occupied by the Bishop disks from outside.
For a unified treatment of the essential analysis for Bishop disks
with boundary on a \LOB{}, see \cite{NiederkrugerHabilitation}.
As in the $3$-dimensional case, we will see that the belt sphere of a
surgery of index $n-1$ on a contact $(2n-1)$-manifold is also
naturally a \LOB{}, so there is again a natural moduli space of
holomorphic disks that fill the belt sphere, implying that it is
nullhomologous.
For surgeries of index $k < n-1$, the belt sphere has dimension
$2n-k-1 > n$, and thus cannot be a \LOB{}, but we will show that after
a suitable deformation, the belt sphere can be viewed as a
parametrized \emph{family} of \LOB{}s, giving rise to a well-behaved
moduli space of disks with moving boundary condition.
It should now be clear why this method cannot work for critical
surgeries: the belt sphere in this case has dimension $n-1$, so it is
too small to define a totally real boundary condition.
The construction of the family of \LOB{}s foliating a general
subcritical belt sphere is somewhat less than straightforward: as we
will see in Section~\ref{sec: description handle and surgery}, the
standard model for a contact form after surgery does not lend itself
well to this construction, but a natural family of \LOB{}s can be
found after deforming to a different model of the belt sphere as
piecewise smooth boundary of a poly-disk.
Let us now discuss the topological reasons why the family of \LOB{}s
foliating $\beltSphere{2n-k-1} \subset (M',\xi')$ in Theorem~\ref{thm:
  main theorem} places constraints on the filling $(W',\omega')$.
We focus for now on the case where $(W',\omega')$ is symplectically
aspherical, which rules out bubbling of holomorphic spheres.
For a suitable choice of tame almost complex structure~$J$
on~$(W',\omega')$, the Bishop families associated to
$\beltSphere{2n-k-1}$ generate a compactified moduli space~$\overline
\mM$ of $J$-holomorphic disks in $W'$ with one marked point, whose
boundaries are mapped to $\beltSphere{2n-k-1}$.
In light of the marked point, this moduli space is necessarily
diffeomorphic to a manifold with boundary and corners of the form
\begin{equation*}
  \overline \Sigma \times \DD^2 \;,
\end{equation*}
where $\overline \Sigma$ is a smooth, compact, connected and oriented
manifold with boundary and corners, whose boundary is a sphere.
Furthermore, the natural evaluation map
\begin{equation*}
  \ev \colon \bigl(\overline\mM, \p\overline \mM\bigr)
  \to \bigl(W',\beltSphere{2n-k-1}\bigr)
\end{equation*}
is smooth, and its restriction
\begin{equation*}
  \restricted{\ev}{\p \overline \mM} \colon 
  \p\overline \mM \to \beltSphere{2n-k-1}
\end{equation*}
is a continuous map of degree~$\pm 1$.
The latter follows readily from the properties of the Bishop family,
which provide a nonempty open subset $U\subset W'$ that intersects
$\beltSphere{2n-k-1}$ and is the diffeomorphic image of $\ev^{-1}(U)
\subset \overline\mM$.
This description of the evaluation map $\ev \colon \overline{\mM} \to
W'$ already implies the homological part of Theorem~\ref{thm: main
  theorem}, i.e.~that $[\beltSphere{2n-k-1}] = 0 \in H_{2n-k-1}(W')$.
To deduce stronger constraints, we will apply two further techniques.
The first consists in performing surgery on the moduli space
$\overline{\mM}$ to simplify its topology and suitably extending the
evaluation map in order to prove $[\beltSphere{2n-k-1}] = 0 \in
\pi_{2n-k-1}(W')$.
This method works when the dimension of $\overline{\mM}$ is not too
large.
The second method is relevant in particular to the case $k=0$ of the
main result, as well as to Theorem~\ref{thm: almost contact connSum
  not genuine}, and is based on the following topological lemma.

\begin{lemma}\label{lemma: map of product manifold induces
    contractibility}
  Let $X,Y$ be smooth orientable compact manifolds with boundary and
  corners such that $\p Y$ is homeomorphic to a sphere and $\dim X + 2
  = \dim Y \ge 3$.
  Write $X' = X\times \DD^2$, and assume that
  \begin{equation*}
    f\colon (X',\p X')\to (Y,\p Y)
  \end{equation*}
  is a continuous map that is smooth on the interior of $X'$, and for
  which we find an open subset~$U\subset \mathring Y$ such that
  $\restricted{f}{f^{-1}(U)} \colon f^{-1}(U) \to U$ is a
  diffeomorphism.
  Then $Y$ is contractible.
\end{lemma}

While it will not be essential to most of our arguments, note that the
$h$-cobordism theorem implies:

\begin{corollary}\label{coro: h-cobordism}
  If $\dim Y \ge 5$ in the lemma, then $Y$ is diffeomorphic to a ball.
\end{corollary}

The $k=0$ case of Theorem~\ref{thm: main theorem} is the case where
$(M',\xi')$ is the standard contact sphere and the belt sphere is the
entirety of~$M'$.
In this setting, applying the above lemma to the evaluation map $\ev
\colon (\overline{\mM} = \overline{\Sigma} \times
\DD^2,\p\overline{\mM}) \to (W',\SS^{2n-1})$ implies that $W'$ must be
diffeomorphic to a ball, hence this reproves the
Eliashberg-Floer-McDuff theorem.
We will explain this argument in more detail in \S\ref{sec:EFM},
including the proof of the lemma (see
Lemma~\ref{lemma:KlausTopology}).
In another context, we will also apply the lemma in \S\ref{sec: almost
  contact surgery not genuine surgery} to demonstrate that the contact
structures arising on the boundaries of certain Stein domains which
look topologically like connected sums cannot arise from index~$1$
contact surgery, thus proving Theorem~\ref{thm: almost contact connSum
  not genuine}.
Here is a brief outline of the paper.
In Section~\ref{sec:EFM}, we provide a foretaste of the methods in the
rest of the paper by using them to give an alternative proof of the
Eliashberg-Floer-McDuff theorem.
Section~\ref{sec: description handle and surgery} then explains the
general case of the family of \LOB{}s associated to a subcritical belt
sphere.
In Section~\ref{sec:moduliSpace}, we define the relevant moduli space
of holomorphic disks and establish its basic properties, leading to
the proof of the homological part of Theorem~\ref{thm: main theorem}.
Section~\ref{sec: surgery on the moduli space} then improves this to a
homotopical statement in cases where the moduli space has sufficiently
low dimension.
Finally, in \S\ref{sec: almost contact surgery not genuine surgery}
and \S\ref{sec:Weinstein} respectively we prove Theorems~\ref{thm:
  almost contact connSum not genuine} and~\ref{thm:Weinstein} on
contact connected sums and contractible Reeb orbits.
The paper concludes with a brief appendix addressing the technical
question of orientability for our moduli space of holomorphic disks.

\section*{Acknowledgments}
The authors would like to thank Yasha Eliashberg, François Laudenbach,
Patrick Massot, Emmy Murphy, Otto van Koert and András Stipsicz for
useful conversations.
They are also very grateful to the referees for careful reading and
suggesting many improvements.
The first and second authors were partially supported during this
project by ANR grant \textsl{ANR-10-JCJC 0102}.
The third author is partially supported by a Royal Society
\textsl{University Research Fellowship} and by EPSRC grant
\textsl{EP/K011588/1}.
Additionally, all three authors gratefully acknowledge the support of
the European Science Foundation's ``CAST'' research network.

\section{The Eliashberg-Floer-McDuff theorem revisited}
\label{sec:EFM}

In this section we modify slightly the proof of the
Eliashberg-Floer-McDuff theorem \cite[Theorem~1.5]{McDuff_contactType}
in order to illustrate the methods that will be applied in the rest of
the article.
The original argument worked by capping off the symplectic filling and
then sweeping through it with a moduli space of holomorphic spheres.
Our version will be the same in many respects, but has more in common
with the $3$-dimensional argument of Eliashberg in
\cite{Eliashberg_filling}: instead of spheres, we use holomorphic
disks attached to a family of \LOB{}s.

\begin{theorem}[Eliashberg-Floer-McDuff]\label{theorem: ElFlMD
    revisited}
  Let $\SS^{2n-1} \subset \CC^n$ be the unit sphere with its standard
  contact structure~$\xi_0$ given by the complex tangencies to the
  sphere, that is,
  \begin{equation*}
    \xi_0 = T\SS^{2n-1} \cap  \bigl(i\cdot T\SS^{2n-1}\bigr) \;.
  \end{equation*}
  Every symplectically aspherical filling of $\bigl(\SS^{2n-1},
  \xi_0\bigr)$ is diffeomorphic to the $(2n)$-ball.
\end{theorem}

Let $\z = \x + i\y = \bigl(x_1+iy_1, \dotsc, x_n+iy_n\bigr)$ be the
coordinates of $\CC^n$.
The function $f\colon \CC^n \to [0,\infty)$ given by
\begin{equation*}
  f(\z) = \sum_{j=1}^n  (x_j^2 + y_j^2)
\end{equation*}
is plurisubharmonic, and the unit sphere is the boundary of the ball
\begin{equation*}
  \DD^{2n} = \bigl\{ \z \in \CC^n\bigm|\, f(\z) \le 1 \bigr\} \;.
\end{equation*}
The main geometric ingredient needed for our proof is a foliation of
$\SS^{2n-1}$ (minus some singular subset) by a family of \LOB{}s, but
this idea does not seem to work when applied directly to the unit
sphere $f^{-1}(1)$.
Instead, we will deform the sphere to a different shape, which does
contain a suitable family of \LOB{}s that will suffice for our
purposes.
Define two functions $g_A, g_B\colon \CC^n \to [0,\infty)$ by
\begin{align*}
  g_A(\z) &= y_1^2 + \dotsm + y_{n-1}^2 \\
  g_B(\z) &= x_1^2 + \dotsm + x_{n-1}^2 + x_n^2 + y_n^2 \;.
\end{align*}
Note that $g_B$ is strictly plurisubharmonic and $g_A$ is weakly
plurisubharmonic as
\begin{align*}
  -dd^c g_A &= 2\, dx_1 \wedge dy_1 + \dotsm + 2\,dx_{n-1} \wedge dy_{n-1} \\
  -dd^c g_B &= 2\, dx_1 \wedge dy_1 + \dotsm + 2\,dx_{n-1} \wedge
  dy_{n-1} + 4\, dx_n \wedge dy_n \;.
\end{align*}
We will now consider the subset (see Fig.~\ref{fig: convex deformation
  sphere})
\begin{equation*}
  \widehat \DD^{2n} = \bigl\{ \z \in \CC^n\bigm|\, g_B(\z) \le 1 \bigr\}
  \cap \bigl\{ \z \in \CC^n\bigm|\, g_A(\z) \le 1 \bigr\} \;.
\end{equation*}
Up to reordering the coordinates, $\widehat \DD^{2n}$ is a
bi-disk~$\DD^{n+1}\times \DD^{n-1} \subset \RR^{2n}$, which clearly
contains the unit ball.
Its boundary is not a smooth manifold, but is nonetheless homeomorphic
to the unit sphere.
It decomposes as
\begin{equation*}
  \p \widehat \DD^{2n} \cong
  \DD^{n+1} \times \bigl(\p \DD^{n-1}\bigr)  \,\cup \,
  \bigl(\p \DD^{n+1}\bigr) \times \DD^{n-1}
  = S_A \cup S_B \;,
\end{equation*}
where we have used the notation
\begin{align*}
  S_A &:= \bigl\{g_A = 1\bigr\} \cap \p \widehat \DD^{2n} \cong
  \DD^{n+1} \times \SS^{n-2}
  \intertext{and}
  S_B &:= \bigl\{g_B = 1\bigr\} \cap \p \widehat \DD^{2n} \cong \SS^n
  \times \DD^{n-1} \;.
\end{align*}

\begin{figure}[htbp]
  \centering
  \includegraphics[height=5cm,keepaspectratio]{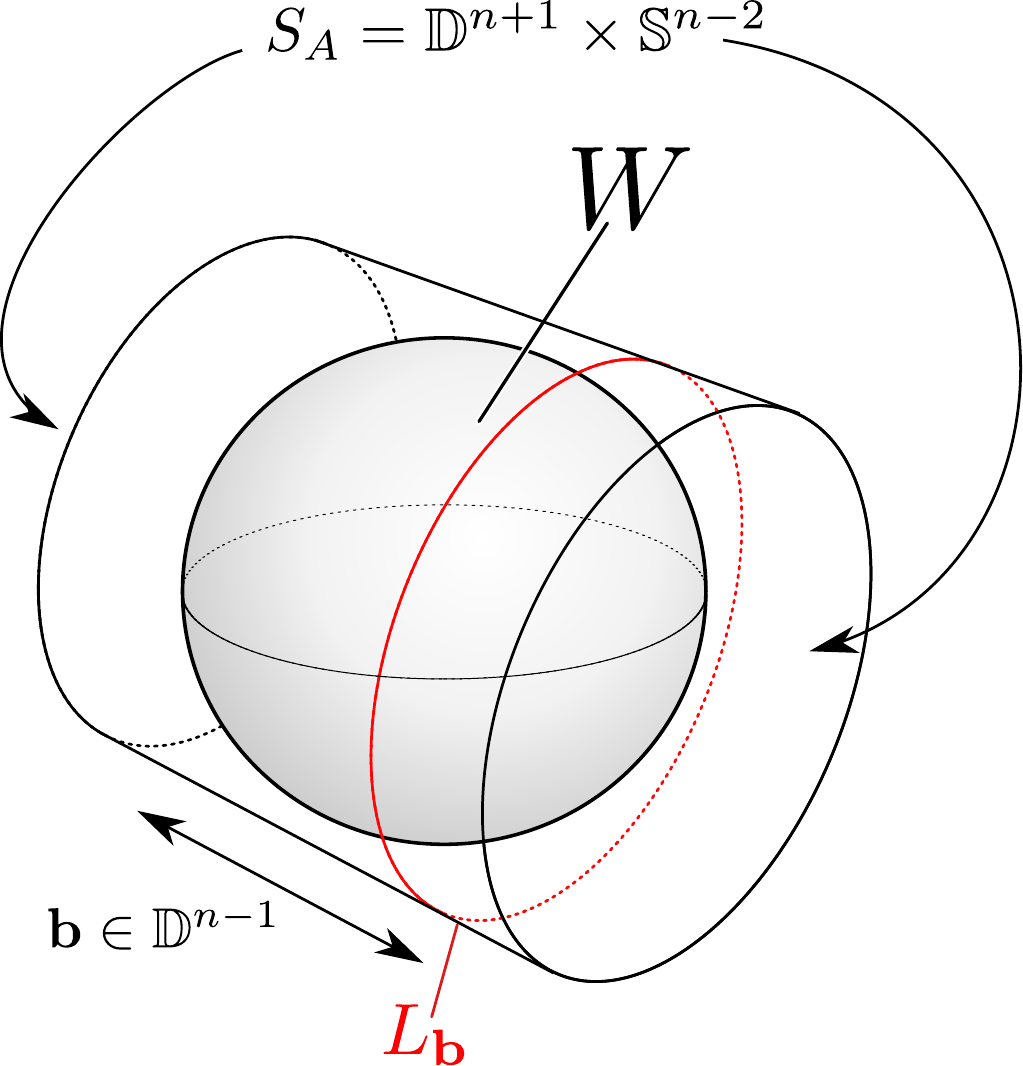}
  \caption{We find a family of \LOB{}s by deforming the sphere to the
    boundary of a bi-disk.  One of the two parts of the boundary,
    which we denote by $S_B$, will then be foliated by
    \LOB{}s.}\label{fig: convex deformation sphere}
\end{figure}

Let now $(W,\omega)$ be a symplectically aspherical filling of
$\bigl(\SS^{2n-1}, \xi_0\bigr)$.
If it is only a weak filling, we can extend it by attaching a
symplectic collar to obtain a strong symplectic filling of the sphere
\cite[Remark~2.11]{WeafFillabilityHigherDimension} because
$\restricted{\omega}{T \SS^{2n-1}}$ is exact.
This filling is diffeomorphic to the initial one, and it is also still
symplectically aspherical, because any $2$-sphere can just be pushed
by a homotopy entirely into the old symplectic filling.
After rescaling the symplectic form, the extended symplectic manifold
will be a strong symplectic filling of the unit sphere.
Remove now the interior~$\mathring \DD^{2n}$ of the unit ball from
$\widehat \DD^{2n}$, and glue $\widehat \DD^{2n} \setminus \mathring
\DD^{2n}$ symplectically onto the filling~$W$.
Denote this new symplectic manifold by $(\widehat W, \widehat
\omega)$.
Clearly $\widehat W$ is homeomorphic to $W$.
Using holomorphic disks, we will show as in the original paper by
McDuff that $\widehat W$ is contractible, so that the $h$-cobordism
theorem \cite{SmaleHCobordism, MilnorHCobordism} implies that $W$ must
be diffeomorphic to $\DD^{2n}$ whenever $2n-1 \ge 5$.
To study $\widehat W$ using holomorphic curves, choose first an almost
complex structure~$J$ on $\widehat W$ that is tamed by $\widehat
\omega$ and that agrees on a small neighborhood of $\p \widehat
\DD^{2n}$ in $\widehat W$ with the standard complex structure~$i$ on
$\CC^n$.
The holomorphic curves we are interested in are attached to a family
of \LOB{}s, which we will introduce now.
Let $\Psi\colon \SS^n \times \DD^{n-1} \to S_B\subset \p\widehat W$ be
the embedding into the boundary of $\widehat W$ given by
\begin{equation*}
  \bigl((a_1,a_2,\dotsc,a_{n+1});(b_1,\dotsc,b_{n-1})\bigr) \mapsto
  \bigl(a_1+ i b_1,\dotsc,a_{n-1} + i b_{n-1},a_n + i a_{n+1}\bigr) \;.
\end{equation*}
The image of $\Psi$ lies in $S_B \subset \p \widehat W$, and the
$J$-complex tangencies on the corresponding part of $\p \widehat W$
are the kernel of the $1$-form
\begin{equation*}
  - d^c g_B = 2x_1\, dy_1 + \dotsm + 2x_{n-1}\, dy_{n-1} +
  2\,\bigl(x_n\, dy_n - y_n\,dx_n\bigr) \;.
\end{equation*}
We obtain for the pull-back
\begin{equation*}
  \Psi^*\bigl(- d^c g_B\bigr) =
  2a_1\, db_1 + \dotsm + 2a_{n-1}\, db_{n-1} +
  2\,\bigl(a_n\, da_{n+1} - a_{n+1}\,da_n\bigr)
\end{equation*}
so that the restriction of $\Psi^*\bigl(- d^c g_B\bigr)$ to each of
the spheres $\SS^n\times\bigl\{(b_1,\dotsc, b_{n-1}) = \mathrm{const}
\bigr\}$ gives
\begin{equation*}
  2\,\bigl(a_n\, da_{n+1} - a_{n+1}\,da_n\bigr) \;.
\end{equation*}
This means that the projection
\begin{equation*}
  \bigl((a_1,a_2,\dotsc,a_{n+1});(b_1,\dotsc,b_{n-1})\bigr) \mapsto
  \arg (a_n + i a_{n+1}) \in \SS^1
\end{equation*}
defines for each $(b_1,\dotsc,b_{n-1})$ a \LOB with the $(n-1)$-ball
as pages and trivial monodromy.
From now on we denote the points in $\DD^{n-1}$ by $\bbf = (b_1,
\dotsc, b_{n-1})$, and write for the \LOB
\begin{equation*}
  L_\bbf = \Psi\bigl(\SS^n \times \{\bbf\} \bigr),
\end{equation*}
and $B_\bbf$ for its binding.
For the technical details of the following part, we refer to
Section~\ref{sec: topology moduli space}.
We will study the space
\begin{equation*}
  \widetilde{\mM}_\star = \Bigl\{(\bbf, u, z_0)\Bigm|\, \bbf \in \DD^{n-1},\,
  u\colon (\DD^2, \p \DD^2) \to \bigl(\widehat W, L_\bbf\bigr),\,
  z_0 \in \DD^2 \Bigr\}
\end{equation*}
of nonconstant holomorphic maps from a disk, equipped with one marked
point $z_0$, and with boundary sent into one of the \LOB{}s~$L_\bbf$.
Additionally, we require that $u$ is homotopic to a Bishop disk as an
element in $\pi_2\bigl(\widehat W, L_\bbf \setminus B_\bbf\bigr)$, and
we denote the corresponding subset by $\widetilde \mM$.
Next we divide $\widetilde \mM$ by the action of the group
$\Aut(\DD^2)$ of biholomorphic transformations on $\DD^2$, where
$\varphi \in \Aut(\DD^2)$ acts on $\widetilde \mM$ via
\begin{equation*}
  \varphi \cdot (\bbf,u,z_0) =
  \bigl(\bbf, u\circ \varphi^{-1}, \varphi(z_0)\bigr) \;.
\end{equation*}
We denote the moduli space $\widetilde \mM / \Aut(\DD^2)$ by $\mM$.
Note that for every class~$[\bbf, u, z]$ in $\mM$, we can fix a unique
representative $(\bbf, u_0, z_0)$ by choosing a parametrization of $u$
such that
\begin{equation}
  u(z) \in
  \begin{cases}
    \text{the $0$ degree page of the \LOB, } & \text{if $z = 1$,} \\
    \text{the $\pi/2$ degree page of the \LOB, } & \text{if $z = i$,} \\
    \text{the $\pi$ degree page of the \LOB, } & \text{if $z = -1$}. \\
  \end{cases}\label{eq: standard parametrization of hol disks}
\end{equation}
A corollary of this is that the moduli space~$\mM$ (before the
compactification, see below) is a \emph{trivial} disk bundle over the
space of unmarked disks.
This is the key fact that will allow us to ``push'' the topology of
$W$ into its boundary (which is the geometric analogue of the
algebraic argument given in \cite{McDuff_contactType} and
\cite{OanceaViterbo}).
Next, we need to understand the compactification of $\mM$.
Note first that typical holomorphic disks are surrounded by a
neighborhood of other typical holomorphic disks, that is, they
represent smooth points of the interior of the moduli space~$\mM$.
With ``typical'', we mean smooth holomorphic disks whose interior
points are mapped to the interior of $\widehat W$, and whose boundary
sits on a \LOB~$L_\bbf$ that is not a boundary \LOB, i.e.~for which
$\norm{\bbf} < 1$, and such that the disk does not touch the
binding~$B_\bbf$ of the \LOB.
Let us now consider the remaining cases.  The boundary of $\widehat W$
consists of $S_A \subset\bigl\{g_A = 1\bigr\}$ and $S_B
\subset\bigl\{g_B = 1\bigr\}$, which are weakly and strongly
plurisubharmonic hypersurfaces respectively.
A disk touching $S_B$ with one of its interior points will
automatically be constant.
If the disk touches $S_A$ instead, then it needs to be entirely
contained in this hypersurface, and in particular its boundary will
lie on a \LOB with $\norm{\bbf} = 1$; below we will explain how to
understand the disks in this second case explicitly.
For every \LOB~$L_\bbf$, there is a certain neighborhood of its
binding~$B_\bbf$ that is only intersected by Bishop disks.
Since there is exactly one disk meeting every point of this
neighborhood, that is, the \defin{evaluation map}
\begin{equation*}
  \ev\colon \mM \to \widehat W, \, [\bbf, u,z_0] \mapsto u(z_0)
\end{equation*}
restricts close to $B_\bbf$ to a diffeomorphism, it follows that the
compactification $\overline{\mM}$ contains disks that collapse to a
point in the binding.
In \cite{NiederkrugerRechtman} it was shown that adding these constant
disks to $\mM$, corresponds to adding points which lie on the smooth
boundary of the compactification~$\overline{\mM}$.
Before understanding the bubbling, we will discuss disks whose
boundary lies in a \LOB~$L_\bbf \subset S_A$.
\begin{lemma}\label{lemma: disks with boundary in S_A}
  Suppose $u \in \overline{\mM}$ maps $\p \DD^2$ to a \LOB{}~$L_\bbf$
  such that $\norm{\bbf}=1$.
  Then the image of $u$ is completely contained in $S_A$, and
  moreover, it is obtained by the intersection of a complex line
  parallel to the $z_n$-plane with $S_A$.
\end{lemma}
\begin{proof}
  Parametrize the disk by polar coordinates~$r e^{i\phi}$.
  By acting on the coordinates $z_1, \dotsc, z_{n-1}$ with a matrix in
  $SO(n-1)$ (regarded as an element of $\SU(n-1)$ with real entries),
  we can assume without loss of generality that the \LOB~$L_\bbf$
  corresponds to the parameter $\bbf = (1,0,\dotsc,0)$, as the
  functions~$g_A$ and $g_B$ are invariant under such an action.
  In particular it follows that the $y_1$-coordinate of $u$ has its
  maximum on the boundary of $u$.
  The $x_1$-coordinate of $\restricted{u}{\p\DD^2}$ is bounded, and
  hence there is an angle $e^{i\phi_0}$ at which the derivative
  \begin{equation*}
    \restricted{\frac{d}{d\phi}}{\phi = \phi_0}
    x_1\Bigl(u\bigl(e^{i\phi}\bigr)\Bigr) = 0
  \end{equation*}
  is zero.
  Complex multiplication gives $i\cdot\partial_r = \partial_\phi$,
  hence
  \begin{equation*}
    dy_1\bigl(Du\cdot\partial_r\bigr)
    = dy_1\bigl(Du\cdot(-i\cdot\partial_\phi)\bigr)
    = - dy_1\bigl(i\cdot Du\cdot\partial_\phi\bigr)
    = - dx_1 \bigl(Du\cdot\partial_\phi\bigr) = 0\;.
  \end{equation*}
  It follows that the outward derivative of the $y_1$-coordinate
  vanishes at the point $e^{i\phi_0} \in \DD^2$, so that according to
  the boundary point lemma, $y_1$ must equal the constant~$1$ on the
  whole disk, and as a consequence $u$ lies entirely in $S_A$.
  The $y_2$- to $y_{n-1}$-coordinates are all $0$ on the boundary of
  the disk, and hence by the maximum principle, they need to be both
  maximal and minimal on all of the disk.
  With the Cauchy-Riemann equation we obtain that the $x_1$- up to
  $x_{n-1}$-coordinates of $u$ need all to be constant on $u$ (for
  more details read Section~\ref{sec: bishop disks at the boundary}).
\end{proof}
As explained in Section~\ref{sec: topology moduli space}, no bubbling
can occur under our assumptions, and hence $\overline{\mM}$ will be a
compact manifold with boundary and corners (the boundary is smooth
everywhere with the exception of the disks corresponding to the edges
of $\widehat W$).
Moreover, the moduli space is orientable (see Appendix~\ref{sec:
  orientability of the moduli space}) and the evaluation map
\begin{equation*}
  \ev\colon \bigl(\overline{\mM}, \p \overline{\mM} \bigr) \to
  \bigl(\widehat W, \p \widehat W\bigr), \,
  [\bbf, u,z_0] \mapsto u(z_0)
\end{equation*}
is a degree~$1$ map, that is, it maps the fundamental class
$[\overline{\mM}] \in H_{2n}\bigl(\overline{\mM}, \p \overline{\mM};
\ZZ\bigr)$ onto the fundamental class $[\widehat W] \in
H_{2n}\bigl(\widehat W, \p \widehat W; \ZZ\bigr)$.
We are therefore in a position to apply the following topological
result, which was stated as Lemma~\ref{lemma: map of product manifold
  induces contractibility} in the introduction.

\begin{lemma}\label{lemma:KlausTopology}
  Let $X,Y$ be smooth orientable compact manifolds with boundary and
  corners such that $\p Y$ is homeomorphic to a sphere and $\dim X + 2
  = \dim Y \ge 3$.
  Write $X' = X\times \DD^2$, and assume that
  \begin{equation*}
    f\colon (X',\p X')\to (Y,\p Y)
  \end{equation*}
  is a continuous map that is smooth on the interior of $X'$, and for
  which we find an open subset~$U\subset \mathring Y$ such that
  $\restricted{f}{f^{-1}(U)} \colon f^{-1}(U) \to U$ is a
  diffeomorphism.
  Then $Y$ is contractible.
\end{lemma}
\begin{proof}
  Note that by Whitehead's theorem, it suffices to show that $Y$ is
  weakly contractible, that is, $\pi_j(Y) = 0$ for all $j > 0$.
  Using Hurewicz's theorem we will show instead that $Y$ is simply
  connected and satisfies $H_j(Y ; \ZZ) = 0$ for all $j > 0$.

  \textbf{(i)} We will first consider the fundamental group of $Y$.
  Choose the base point~$p_0 \in U\subset \mathring Y$.
  Let $\gamma$ be a smooth, embedded loop representing a class in
  $\pi_1\bigl(Y, p_0\bigr)$ that lies in the interior of $Y$.
  After a perturbation, we can assume that $\gamma$ is transverse to
  the map~$f$, so $f^{-1}\bigl(\gamma\bigr)$ will be a finite
  collection of loops~$\Gamma_0,\dotsc,\Gamma_N$ in $X'$.
  There is one loop, say $\Gamma_0$, that is mapped to $\gamma$ with
  degree one.
  The reason for this is that $\gamma$ runs through $U$, where $f$ is
  a diffeomorphism.
  Then the loop $f\circ \Gamma_0$ is homotopic to $\gamma$, and thus
  represents the same class in $\pi_1\bigl(Y, p_0\bigr)$.
  Using the fact that $X'$ is diffeomorphic to a trivial disk bundle,
  we may shift $\Gamma_0$ into the boundary of $X'$ (just by moving it
  inside the $\DD^2$-factor).
  In particular, this shows that $[\gamma] = [f\circ \Gamma_0]$ can be
  represented by a loop that lives in the boundary of $Y$, and is thus
  contractible.
  \textbf{(ii)} Next we need to compute the homology of $Y$.
  It is easy to see that the image of
  \begin{equation*}
    f_*\colon H_*\bigl(X'; \ZZ\bigr) \to H_*\bigl(Y; \ZZ\bigr)
  \end{equation*}
  is trivial.
  Indeed, all homology groups $H_k\bigl(X'; \ZZ\bigr)$ with $k \ge
  \dim X$ are trivial, so that we only need to study $k < \dim X <
  \dim \p Y$.
  Let $A\in H_k(Y; \ZZ)$ be a homology class that lies in the image of
  $f_*$ so that there is a $B \in H_k(X'; \ZZ)$ with $A = f_*B$.
  Since $X' = X \times \DD^2$ where $\DD^2$ is contractible, $B$ can
  be represented by a cycle in $X \times \{p\}$ for any point $p \in
  \DD^2$; in particular we are free to choose $p \in \p\DD^2$, hence
  $B$ is represented by a cycle in~$\p X'$.
  This implies that the class~$A$ is homologous to a cycle in the
  sphere~$\p Y$, which shows that $A$ must be trivial.
  We will now show that $f_*\colon H_k\bigl(X'; \ZZ\bigr) \to
  H_k\bigl(Y; \ZZ\bigr)$ is surjective for every $k \le
  \frac{1}{2}\,\dim Y$.
  Assume for now that $k< \frac{1}{2}\,\dim Y$ and that we already
  have shown for every $l<k$ that $H_l(Y; \ZZ) = 0$.
  From (i), we see that $k \ge 2$.
  By Hurewicz's theorem we know that $H_k\bigl(Y; \ZZ\bigr)$ is
  isomorphic to $\pi_k\bigl(Y, p_0\bigr)$ so that we can represent any
  homology class in $H_k\bigl(Y; \ZZ\bigr)$ by a map
  \begin{equation*}
    s\colon \SS^k \to Y \;.
  \end{equation*}
  We additionally assume again that the base point~$p_0$ lies in $U$.
  After a generic perturbation (see
  \cite[Theorem~II.2.12]{HirschDiffTopology}), $s$ will be an
  embedding that is transverse to $f$, and we find a closed (possibly
  disconnected) smooth submanifold $f^{-1}(s) \subset X'$.
  As in step~(i), there is a unique connected component~$S$ of $f^{-1}
  (s)$ which passes through the base point~$p_0$.
  Clearly we have a map $S \to \SS^k$ of degree one that makes the
  following diagram commute:
  \begin{equation*}
    \xymatrix{
      S \ar[rr] \ar[dr]_{f} & & \SS^k \ar[dl]^{s} \\
      & Y & 
    }
  \end{equation*}
  As we wanted to show, it follows that $f(S)$ represents the same
  class in $H_k\bigl(Y; \ZZ\bigr)$ as $s$, so that $f_*\colon
  H_k\bigl(X'; \ZZ\bigr) \to H_k\bigl(Y; \ZZ\bigr)$ is surjective, and
  as a consequence $H_k\bigl(Y; \ZZ\bigr)$ is trivial.
  Let us now study the case $k = \frac{1}{2}\,\dim Y$.
  Again we may represent any element in $H_k\bigl(Y; \ZZ\bigr)$ by a
  map
  \begin{equation*}
    s\colon \SS^k \to Y
  \end{equation*}
  that goes through $p_0 \in U$.
  Perturbing $s$, we may assume that it is an immersion with
  transverse self-intersections, and also that it is transverse to $f$
  (position the double-points at regular values of $f$).
  The map $F := (f,s)\colon X'\times \SS^k \to Y \times Y$ is
  transverse to the diagonal $\triangle_Y \subset Y\times Y$, and it
  follows that $F^{-1}(\triangle_Y)$ is a closed smooth submanifold of
  $X'\times \SS^k$.
  Let $S$ be the unique component of $F^{-1}(\triangle_Y)$ that is
  mapped by $F$ to $(p_0,p_0)$.
  By definition we have $f\circ \restricted{\Pi_1}{S} = s\circ
  \restricted{\Pi_2}{S}$, where $\Pi_1\colon X'\times \SS^k \to X'$
  and $\Pi_2\colon X'\times \SS^k \to \SS^k$ are the canonical
  projections.
  Furthermore $\Pi_2$ is a degree one map.
  This ends the proof, because $[f\circ \restricted{\Pi_1}{S}] = [s]
  \in H_k(Y; \ZZ)$.
  \textbf{(iii)} Above we have seen that $H_k\bigl(Y; \ZZ\bigr)$ is
  trivial for all $1\le k \le \frac{1}{2}\, \dim Y$.
  Using Poincaré-Lefschetz duality, $H^{\dim Y - k}\bigl(Y, \p Y;
  \ZZ\bigr)$ is isomorphic to $H_k\bigl(Y; \ZZ\bigr)$ and thus it also
  vanishes (for $1 \le k \le \frac 12 \dim Y$).
  The long exact sequence of the pair then implies that $H^{\dim Y -
    k}\bigl(Y; \ZZ\bigr) = \{0\}$, and finally the universal
  coefficient theorem tells us that $H_{\dim Y - k}\bigl(Y; \ZZ\bigr)
  = \{0\}$, so that $Y$ has trivial homology.
\end{proof}

To conclude the proof of Theorem~\ref{theorem: ElFlMD revisited},
simply apply the lemma and Corollary~\ref{coro: h-cobordism} to the
moduli space and its evaluation map found above.
It follows that $\widehat W$ is diffeomorphic to a ball.

\section{Weinstein handles and contact surgeries}\label{sec:
  description handle and surgery}

In this section, we will give a precise description of Weinstein
handle attachment, and then show how to deform the contact structure
near the belt sphere in order to find a suitable family of \LOB{}s.
Assume $(W,\omega)$ is a weak symplectic filling of a
$(2n-1)$-dimensional contact manifold $(M,\xi)$, choose an isotropic
sphere $\attachingSphere{k-1} \hookrightarrow (M,\xi)$ with trivial
normal bundle~$\nu \attachingSphere{k-1}$, and suppose that
$\restricted{\omega}{T \attachingSphere{k-1}}$ is exact (this is
obviously always the case if $k \ne 3$).
After deforming the symplectic structure in a small neighborhood of
the boundary using \cite[Remark~2.11]{WeafFillabilityHigherDimension},
we can find a Liouville vector field~$X_L$ on a neighborhood $U
\subset W$ close to $\attachingSphere{k-1}$ that points transversely
out of $W$, and such that
\begin{equation*}
  \restricted{\alpha}{U \cap  M} =
  \restricted{(\iota_{X_L} \omega)}{T(U \cap  M)}
\end{equation*}
is a contact form for $\xi$.
Topologically, the handle attachment can be described as follows.
Choose a trivialization of the normal bundle $\nu
\attachingSphere{k-1}$ in $M$, identifying a tubular neighborhood with
\begin{equation*}
  T^*\SS^{k-1} \times \RR^{2(n-k) +1} \cong \SS^{k-1} \times \RR^{2n - k} \;,
\end{equation*}
and let the $k$-handle~$H_k$ be the poly-disk
\begin{equation*}
  H_k := \DD^k \times \DD^{2n-k} \;.
\end{equation*}
The boundary of $H_k$ can be written as the union
\begin{equation*}
  \p H_k := \SS^{k-1} \times \DD^{2n-k} \, \cup  \,
  \DD^k \times \SS^{2n-k-1} \;.
\end{equation*}
Using the obvious homeomorphism, we can ``glue'' $H_k$ along the
subset $\SS^{k-1} \times \DD^{2n-k} \subset \p H_k$ to the
neighborhood of $\attachingSphere{k-1}$ using the trivialization
chosen above.
Denote the new manifold by
\begin{equation*}
  W' = W \cup_{\nu \attachingSphere{k-1}} H_k \;.
\end{equation*}
Note that the gluing operation depends on the trivialization chosen
for the normal bundle of $\attachingSphere{k-1}$.
The boundary $M' = \p W'$ is obtained from the old contact
manifold~$M$ by removing the neighborhood of $\attachingSphere{k-1}$,
and gluing in the free boundary component of the handle, that is,
\begin{equation*}
  M' = \bigl(M \setminus \nu\attachingSphere{k-1}\bigr)
  \cup  (\DD^k \times \SS^{2n-k-1}) \;.
\end{equation*}
This operation changing $M$ to $M'$ is called a \defin{contact surgery
  of index~$k$ along $\attachingSphere{k-1}$}.
We will recall below how the natural symplectic structure on $W'$ with
weakly contact-type boundary $M'$ is defined.
The \defin{belt sphere}~$\beltSphere{2n-k-1}$ of the handle~$H_k$ is
the ``cosphere to the gluing sphere'',
\begin{equation*}
  \beltSphere{2n-k-1} := \{\0\} \times \SS^{2n-k-1} \subset \p H_k \;.
\end{equation*}
Note that contact surgery can also be defined as an operation on
contact manifolds without assuming that they are symplectically
fillable:
one only need regard $(M,\xi)$ as the contact boundary of a piece of
its symplectization $\bigl( (-\epsilon,0] \times M, d(e^t \alpha)
\bigr)$.
Topologically, a surgery of index~$k$ is the operation of removing a
small neighborhood of the isotropic sphere~$\attachingSphere{k-1}$
from the contact manifold~$(M,\xi)$ and gluing into the cavity a
standard patch that is diffeomorphic to $\DD^k \times \SS^{2n-k-1}$.

\subsection{Attaching Weinstein handles}
\label{sec: description of handle attachments}

We now define a symplectic model for a subcritical Weinstein handle of
index $k$ in dimension~$2n$.
For this, split $\CC^n$ into $\CC^k \times \CC^m \times \CC$ with $n =
k + m + 1$, and write the coordinates on $\CC^k$ as
\begin{align*}
  \z^- = \x^- + i \, \y^- &:= \bigl(x^-_1 + iy^-_1,\dotsc,x^-_k +
  iy^-_k\bigr) \;,
  \intertext{the ones on $\CC^m$ as}
  \z^+ = \x^+ + i\, \y^+ &:= \bigl(x^+_1 + i y^+_1,\dotsc,x^+_m + i
  y^+_m) \;,
  \intertext{and the ones on $\CC$ as}
  z^\circ = x^\circ + i\, y^\circ \;.
\end{align*}
The coordinate~$z^\circ$ behaves like any other of the coordinates in
$\z^+$, and in the usual descriptions of the handle attachment, it is
not distinguished from $\z^+$.  (For a critical handle 
attachment $k=n$, thus there are no $\z^+$- or $z^\circ$-coordinates.)
The reason why we have introduced this more complicated notation is to
prepare for the deformation we will perform in the next section,
in which the $z^\circ$-coordinate will play a particular role.
As a model for the handle, take
\begin{equation*}
  H_r := \DD^k \times \DD^{2n-k}_r \subset \CC^n \;,
\end{equation*}
where the first disk corresponds to the $\y^-$-coordinates, the second
disk to the $(\x^-, \z^+, z^\circ)$-coordinates, and $r > 0$ is a
constant, i.e.~$H_r$ is the intersection of the two subsets
\begin{equation*}
  \bigl\{\norm{\y^-}^2 \le 1\bigr\} \,\cap \,
  \bigl\{\norm{\x^-}^2 + \norm{\z^+}^2 + \abs{z^\circ}^2 \le r^2\bigr\}\;.
\end{equation*}
We denote
\begin{equation*}
  \p_- H_r = \p \DD^k \times \DD^{2n-k}_r, \quad
  \p_+ H_r =  \DD^k \times \p \DD^{2n-k}_r \;.
\end{equation*}
The core of $\p_-H_r$ is the $(k-1)$-sphere
\begin{align*}
  S_- &:= \bigl\{\norm{\y^-}^2 = 1,\, \x^- = \0, \, \z^+ = \0,\, z^\circ =
  0\bigr\}
  \intertext{which will be identified with the attachment sphere in a
    contact manifold; the core of $\p_+H_r$ is the $(n+m)$-sphere
    (note that $n+m = 2n -k -1$)}
  S_+ &:= \bigl\{\y^- = \0,\, \norm{\x^-}^2 + \norm{\z^+}^2 +
  \abs{z^\circ}^2 = r^2\bigr\} \;.
\end{align*}
Choose on $\CC^n$ the symplectic form
\begin{equation*}
  \omega = 2\, \sum_{r=1}^k dx^-_r \wedge  dy^-_r
  +   4\, \sum_{s=1}^m dx^+_s \wedge  dy^+_s
  +  4\, dx^\circ \wedge  dy^\circ \;.
\end{equation*}
It admits the Liouville form
\begin{equation*}
  \lambda = 2\, \sum_{r=1}^k \bigl(2x^-_r\, dy^-_r + y^-_r \, dx^-_r\bigr)
  + 2\, \sum_{s=1}^m \bigl(x^+_s \, dy^+_s  - y^+_s \, dx^+_s\bigr)
  + 2\, (x^\circ \, dy^\circ - y^\circ \, dx^\circ)
\end{equation*}
that is associated to the Liouville vector field
\begin{equation*}
  \begin{split}
    X_L &= \sum_{r=1}^k \Bigl(2x_r^-\, \frac{\partial}{\partial
      x_r^-} - y_r^-\, \frac{\partial}{\partial y_r^-} \Bigr) \\
    & \qquad + \frac{1}{2}\, \sum_{s=1}^m \Bigl(x_s^+\,
    \frac{\partial}{\partial x_s^+} + y_s^+\, \frac{\partial}{\partial
      y_s^+}\Bigr) +
    \frac{1}{2}\,\Bigl(x^\circ\,\frac{\partial}{\partial x^\circ} +
    y^\circ\, \frac{\partial}{\partial y^\circ}\Bigr) \;.
  \end{split}
\end{equation*}
The field~$X_L$ points outward through $\p_+H_r$ and inward at
$\p_-H_r$, so that both $\p_+H_r$ and $\p_-H_r$ are contact type
hypersurfaces with the corresponding coorientations.
The core~$S_- \subset \p_-H_r$ is an isotropic sphere with trivial
conformal symplectic normal bundle.

\begin{figure}[htbp]
  \centering
  \includegraphics[height=5cm,keepaspectratio]{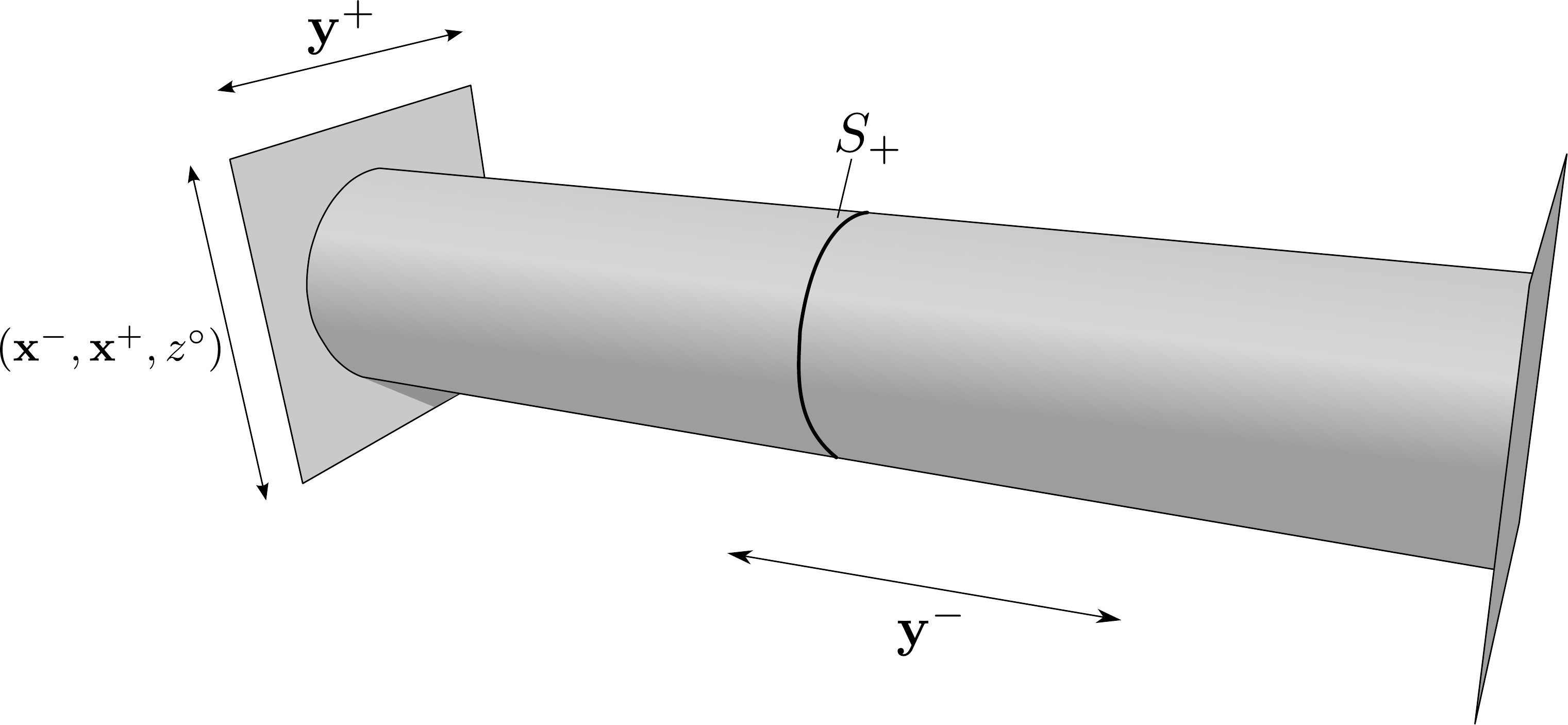}
  \caption{The handle can be glued onto a contact
    manifold.}\label{fig: unmodified handle}
\end{figure}

Let now $(M,\xi)$ be a given contact manifold and let
$\attachingSphere{k-1} \subset M$ be a $(k-1)$-dimensional isotropic
sphere with trivial conformal symplectic normal bundle that will serve
as the \defin{attaching sphere} of the $k$-handle~$H_r$.
Fixing $r > 0$ small enough, $\p_- H_r$, endowed with the contact
structure induced by $\lambda$, is contactomorphic to a
neighborhood~$\nN(\attachingSphere{k-1}) \subset (M,\xi)$ of
$\attachingSphere{k-1}$.
We choose a contact form $\alpha$ for $\xi$ on $M$ such that
$\restricted{\alpha}{\nN(\attachingSphere{k-1})}$ can be glued to
$\restricted{\lambda}{T(\p_-H_r)}$ and define the Liouville manifold
\begin{equation*}
  (W_0,\lambda_0) := \bigl((-\epsilon,0]\times M, e^t\alpha\bigr)
  \cup_{\p_-H_r} (H_r,\lambda) \;.
\end{equation*}
The positive boundary of $W_0$ (denoted $\p_+ W_0$) has two smooth
faces $M \setminus \nN(\attachingSphere{k-1})$ and $\p_+H_r$, meeting
along a corner which is the image of the corner $\p \DD^k \times \p
\DD^{2n-k}_r$ in $\p H_r$, see Fig.~\ref{fig: unmodified handle}.
Fix a small neighborhood~$\uU$ of the corner, and choose a smooth
hypersurface~$M'$ that matches $\p_+ W_0$ outside of $\uU$, and is
transverse to $X_L$ in $\uU$.
Denote the induced contact structure on $M'$ by $\xi' = TM' \cap \ker
\lambda_0$.
Note that the constant~$r > 0$ can be made arbitrarily small, without
changing the isotopy class of the contact structure on $M'$; we can
shrink the size of the handle continuously (including the smoothing)
which allows us to apply Gray stability to obtain an isotopy with
support in the model neighborhood.
The belt sphere~$\beltSphere{2n-k-1} = \beltSphere{n + m}$ of the
$k$-handle is the core~$S_+$ of $\p_+ H_r$.

\subsection{Families of \LOB{}s on a deformed subcritical handle}
\label{sec: deformed subcritical handles}

To find the desired family of \LOB{}s, we will now modify the contact
structure in a neighborhood of the belt sphere in two steps.
The first deformation is borrowed from the recent article
\cite{GeigesZehmisch}; it replaces a technically more complicated
method that was used in an earlier version of this paper.
Consider again a ``thin'' handle $H_r = \DD^k \times \DD^{2n-k}_r
\subset \CC^n$ with $r \ll 1$ as used above.
Suppose that the rounding of the corners has been performed for values
of $\norm{\y^-}$ in the interval~$[1 - \epsilon, 1]$.
The part of $\p_+ H_r$ outside the smoothing region lies in the level
set~$\{f = r^2\}$ of the function
\begin{equation*}
  f(\z^-,\z^+, z^\circ) = \norm{\x^-}^2 + \norm{\z^+}^2 + \abs{z^\circ}^2 \; .
\end{equation*}
We would like to modify the Liouville field on a neighborhood of the
belt sphere so that the induced contact structure coincides with the
field of complex hyperplanes on the boundary.
For this, add the Hamiltonian vector field~$X_H$ of a function
$H\colon \CC^n \to \RR$ to $X_L$, since then $\hat X_L := X_L + X_H$
will still be a Liouville field.
Let $\rho\colon [0,\infty) \to [-1,0]$ be a smooth function that is
equal to $-1$ on the interval $[0, \sqrt{1-2\epsilon}]$, equal to $0$
on the interval $[\sqrt{1-\epsilon}, 1]$ and that increases
monotonically in between.
Define the Hamiltonian function
\begin{equation*}
  H(\z^-, \z^+, z^\circ) := 2\, \langle\x^-, \y^-\rangle\,
  \rho\bigl(\norm{\y^-}^2\bigr) \;.
\end{equation*}
The Hamiltonian vector field corresponding to $H$ is
\begin{equation*}
  X_H = - \sum_{r=1}^k y^-_r\, \rho\bigl(\norm{\y^-}^2\bigr)\, 
  \frac{\partial}{\partial y^-_r}  +
  \sum_{r=1}^k \left(2\,\langle\x^-, \y^-\rangle\, y^-_r\, 
    \rho'\bigl(\norm{\y^-}^2\bigr)
    + x^-_r\, \rho\bigl(\norm{\y^-}^2\bigr) \right)\,
  \frac{\partial}{\partial x^-_r} \;.
\end{equation*}
The vector field $\hat X_L$ agrees outside the support of $\rho$ with
$X_L$, and it is everywhere transverse to $\p_+ H_r$ as can be seen
from
\begin{equation*}
  \begin{split}
    \lie{\hat X_L} f &= \lie{X_L} f + \lie{X_H} f \\
    & = 4\, \norm{\x^-}^2 + \norm{\z^+}^2 + \abs{z^\circ}^2 +
    \lie{X_H} \norm{\x^-}^2 \\
    & = (4 + 2\, \rho)\, \norm{\x^-}^2 + \norm{\z^+}^2 +
    \abs{z^\circ}^2 + 4\,\langle\x^-, \y^-\rangle^2\, \rho' > 0 \;,
  \end{split}
\end{equation*}
because $\rho \ge -1$, and $\rho' \ge 0$.
It follows that $\lambda$ and $\hat \lambda := \iota_{\hat X_L}
\omega$ induce isotopic contact structures on $M'$.
The contact structure on the domain $\p_+ H_r \cap \{\norm{\y^-}^2 \le
1 - 2\epsilon\}$ is the kernel of the Liouville form
\begin{equation*}
  \hat \lambda = \lambda  + dH =  2\, \sum_{r=1}^k x^-_r\, dy^-_r 
  + 2\, \sum_{s=1}^m \bigl(x^+_s \, dy^+_s  - y^+_s \, dx^+_s\bigr)
  + 2\, (x^\circ \, dy^\circ - y^\circ \, dx^\circ)  \;.
\end{equation*}

\begin{remark}\label{rmk: always long cylinder}
  This first deformation shows that the surgered manifold contains a
  neighborhood of the belt sphere that is contactomorphic to a
  cylinder $\bigl\{f = r^2\} \cap \bigl\{\norm{\y^-}^2 \le 1/2\bigr\}
  \subset \CC^n$ with $r$ arbitrarily small and a contact structure
  given as kernel of $\hat \lambda$.
  Note that $\hat \lambda$ on the domain under consideration is equal
  to the differential~$- d^c f$, i.e.~the contact structure on our
  domain coincides with the complex tangencies.
  This is the key fact that we will exploit in the second deformation
  below.
\end{remark}

\vspace{1cm}

To continue, we consider the setting of Theorem~\ref{thm: main
  theorem}, in which $(M',\xi')$ was a fillable contact manifold
obtained by subcritical surgery.
Since this will be the main object of study from now on, it will be
convenient to simplify the notation, hence we assume (unlike in the
statement of Theorem~\ref{thm: main theorem}) that $(M,\xi)$ is a
closed contact manifold of dimension $2n-1$ that has been obtained by
a surgery of index $k \in \{1,\dotsc,n-1\}$ from another contact
manifold, and let $(W, \omega)$ be a weak symplectic filling of $(M,
\xi)$.
Since the theorem in dimension three already follows from the much
stronger result of Eliashberg \cite{Eliashberg_filling}, we are free
to assume $n \ge 3$.  The belt sphere then has dimension $2n-k-1 \ge n
\ge 3$, hence the restriction of $\omega$ to $\beltSphere{2n-k-1}$ is
automatically exact.
It follows (using \cite[Remark~2.11]{WeafFillabilityHigherDimension})
that $\omega$ can be deformed in a collar neighborhood of $\p W$ so
that an outwardly transverse Liouville vector field exists near
$\beltSphere{2n-k-1}$, and we are therefore free to pretend in the
following discussion that $(W,\omega)$ is a \emph{strong} filling of
$(M,\xi)$.
In particular, we may assume that the symplectic structure on a collar
neighborhood close to the belt sphere looks like the symplectic
structure on the boundary of the model of the handle, and we may
identify both.
Let $f\colon \CC^n \to [0,\infty)$ be again the plurisubharmonic
function
\begin{equation*}
  f(\z^-,\z^+, z^\circ) = \norm{\x^-}^2 + \norm{\z^+}^2 + \abs{z^\circ}^2 \; .
\end{equation*}
By the explanations above (see Remark~\ref{rmk: always long
  cylinder}), the belt sphere~$\beltSphere{n + m}$ has a
neighborhood~$\uU_M \subset M$ that is contactomorphic to the cylinder
\begin{equation*}
  C_r := \bigl\{(\z^-, \z^+, z^\circ) \in \CC^n\bigm|\,
  f(\z^-,\z^+, z^\circ) = r^2, \, \norm{\y^-}^2 < 1/2  \bigr\}
\end{equation*}
for arbitrarily small $r \ll 1$ with contact structure~$\hat{\xi}$
given as the kernel of the Liouville form~$\hat{\lambda} =
-\restricted{d^cf}{TC_r}$, and since the cylinder is a level set of
$f$, this also means that $\hat{\xi}$ are the complex tangencies of
$C_r$.
We denote by $\uU_W$ a small neighborhood of $\uU_M$ in $W$ that is
symplectomorphic to the subset
\begin{equation*}
  \bigl\{(\z^-, \z^+, z^\circ) \in \CC^n\bigm|\,
  f(\z^-,\z^+, z^\circ) \in (r^2-\delta,r^2], \, \norm{\y^-}^2 < 1/2 \bigr\}
\end{equation*}
with symplectic form $\omega = -dd^c f$.
Using the embedding of $\uU_W$ into our model, we can extend the
symplectic filling $(W, \omega)$ by attaching the following compact
symplectic subdomain of $\CC^n$:
replace $f$ by $G := \max\{g_A, g_B\}$ which is obtained as the
maximum of the two functions:
\begin{align*}
  g_A(\z^-,\z^+, z^\circ) &:= \norm{\y^+}^2 \\
  \intertext{and}
  g_B(\z^-,\z^+, z^\circ) &:= \norm{\x^-}^2 + \norm{\x^+}^2 +
  \psi\bigl(\norm{\y^-}\bigr) \cdot \norm{\y^+}^2 + \abs{z^\circ}^2
  \;,
\end{align*}
where $\psi$ is a cut-off function that vanishes close to $0$, and
increases until it reaches $1$ close to $\norm{\y^-} = 1/\sqrt{2}$.
Clearly $g_A$ is a weakly plurisubharmonic function.
The function~$g_B$ is strictly plurisubharmonic on a neighborhood of
$\{\y^+ = \0\}$ because the last term of
\begin{equation*}
  -dd^cg_B = 2\, \sum_{r=1}^k dx^-_r \wedge dy^-_r
  + 2\, \sum_{s=1}^m dx^+_s \wedge dy^+_s + 4\, dx^\circ \wedge dy^\circ
  - dd^c \Bigl(\psi\bigl(\norm{\y^-}\bigr) \cdot \norm{\y^+}^2\Bigr)
\end{equation*}
simplifies along this subset to
\begin{equation*}
  \begin{split}
    -dd^c\Bigl(\psi\bigl(\norm{\y^-}\bigr) \cdot \norm{\y^+}^2\Bigr)
    &= - \psi\bigl(\norm{\y^-}\bigr) \, dd^c\norm{\y^+}^2
    - d\norm{\y^+}^2\wedge d^c\psi\bigl(\norm{\y^-}\bigr)  \\
    &\qquad\qquad - d\psi\bigl(\norm{\y^-}\bigr) \wedge
    d^c\norm{\y^+}^2 -\norm{\y^+}^2\, dd^c\psi\bigl(\norm{\y^-}\bigr) \\
    &= 2 \psi\bigl(\norm{\y^-}\bigr) \, \sum_{s=1}^m dx^+_s \wedge
    dy^+_s \; ,
  \end{split}
\end{equation*}
which is weakly plurisubharmonic.
This implies that if the chosen handle~$C_r$ is thin enough, that is
if $r > 0$ has been chosen sufficiently small, then $g_B$ will be
strictly plurisubharmonic on its neighborhood.
For large values of $\norm{\y^-}^2$, the cut-off function is equal to
$1$ and $g_B$ agrees with $f$.
Since it also dominates $g_A$, the level set~$\{G = r^2\}$ glues
smoothly to the given contact manifold, and it bounds a symplectic
manifold
\begin{equation*}
  \widehat W = W \cup  \bigl\{ (\z^-,\z^+, z^\circ) \in \CC^n\bigm|\,
  f(\z^-,\z^+, z^\circ) \ge r^2 \text{ and }
  G (\z^-,\z^+, z^\circ) \le r^2\bigr\}
\end{equation*}
obtained from the given symplectic filling~$W$ by attaching to it the
symplectic domain lying in our model between the level sets~$\{f =
r^2\}$ and $\{G = r^2\}$, see Figures~\ref{fig: deformed handle} and
\ref{fig: deformed handle cut open}.
Note that the boundaries of $W$ and $\widehat W$ are continuously
isotopic.
We also write $\Wmodel$ for the subdomain $\uU_W \cup \bigl(\widehat W
\setminus W\bigr)$ that lies entirely in $\CC^n$.
We decompose the boundary of $\widehat W$ into three domains, which we
denote by $\Mreg$, $M_A$ and $M_B$.
Here $M_A$ and $M_B$ are the parts of $\p \widehat W$ that lie in the
level set $\{g_A = r^2\}$ or $\{g_B = r^2\}$ respectively, and satisfy
$\norm{\y^-}^2 < 1/2$; $\Mreg$ is the remaining part of the boundary
of $\widehat W$, i.e.~the part that is disjoint from the boundary of
the deformed handle.
The boundary of the handle contains a deformation of the belt sphere,
which we will still write as $\beltSphere{n + m} = \{\y^- = \0\}$ even
though it has edges.
The cut-off function~$\psi$ vanishes on a neighborhood of
$\beltSphere{n+m}$, so that $G$ simplifies to
\begin{equation*}
  G (\z^-,\z^+, z^\circ) :=  \max\Bigl\{\norm{\y^+}^2,\,
  \norm{\x^-}^2 +  \norm{\x^+}^2 + \abs{z^\circ}^2 \Bigr\}\;.
\end{equation*}
It follows that $\beltSphere{n+m}$ is the boundary of the poly-disk
\begin{equation*}
  \Bigl\{\norm{\y^+}^2 \le r^2, \, \y^- = \0\Bigr\} \cap 
  \Bigl\{ \norm{\x^-}^2 + \norm{\x^+}^2 + \abs{z^\circ}^2 \le r^2,
  \y^- = \0\Bigr\} \cong \DD^m \times \DD^{n+1} \;.
\end{equation*}

\begin{figure}[htbp]
  \centering
  \includegraphics[height=5cm,keepaspectratio]{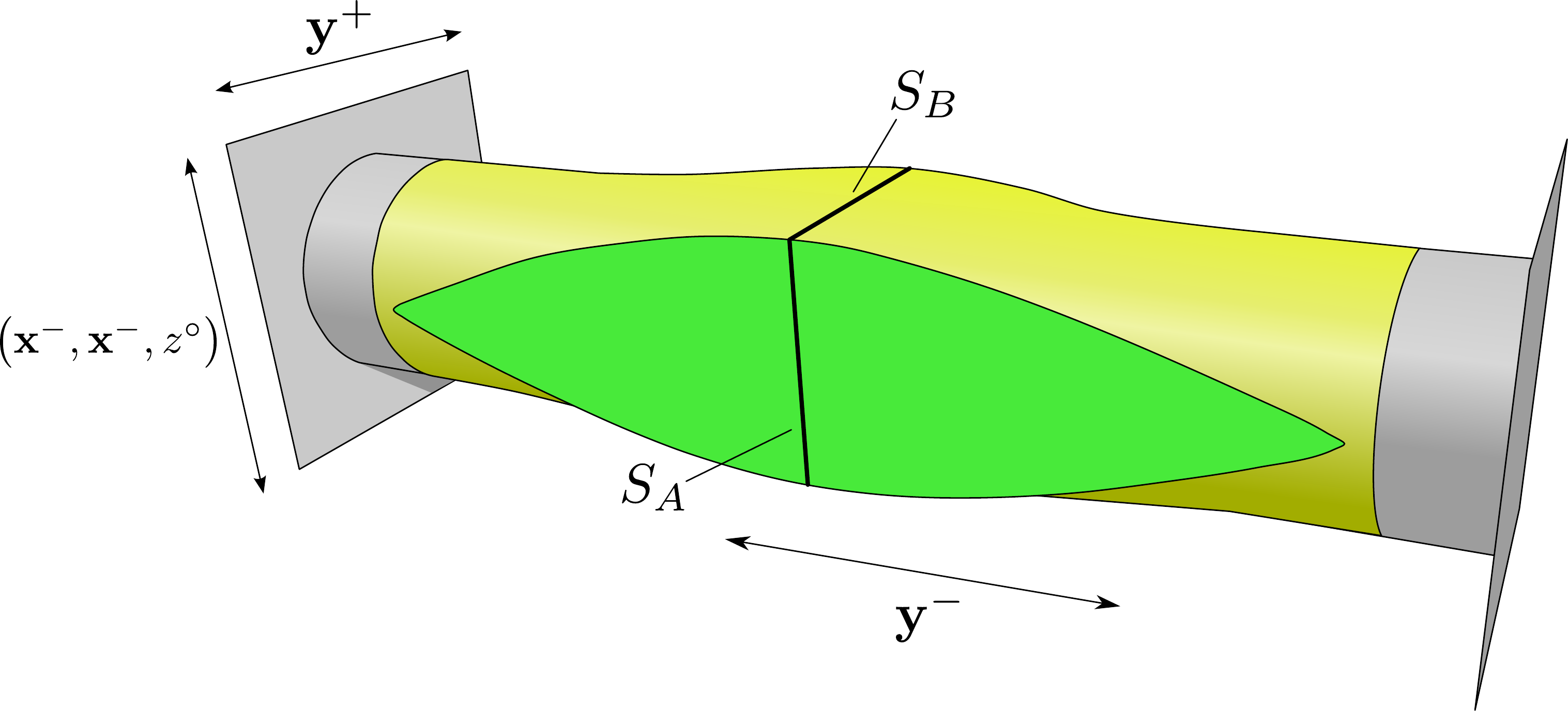}
  \caption{We need to deform the handle to find a suitable family
    of \LOB{}s  in the belt sphere.
    The new handle will have edges.
    The green area in the picture represents the boundary~$M_A$, the
    yellow one is the boundary~$M_B$, and the grey part is
    $\Mreg$.
    The  belt sphere~$S_A\cup S_B$ corresponding to the deformed
    handle also has edges.
    The part~$S_B$ is foliated by \LOB{}s.}\label{fig: deformed
    handle}
\end{figure}

\begin{figure}[htbp]
  \centering
  \includegraphics[height=5cm,keepaspectratio]{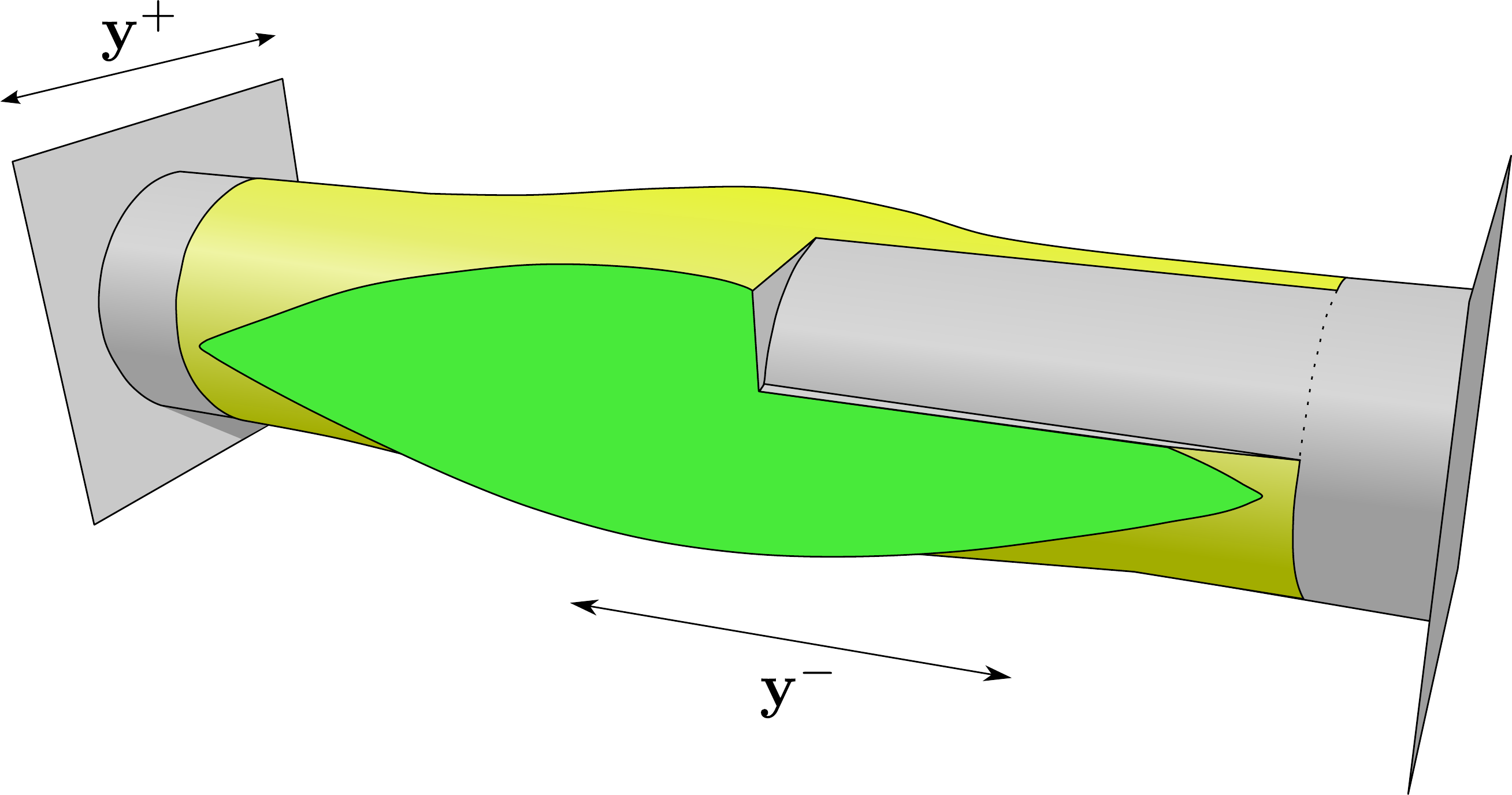}
  \caption{The deformed handle differs from the original one by
      the attachment of a symplectic cobordism.
      This cobordism can also be added to any other symplectic filling
      of the surgered contact manifold.
      The figure shows a cut through this cobordism.  }\label{fig:
      deformed handle cut open}
\end{figure}

We can decompose the boundary of the poly-disk $\DD^m \times
\DD^{n+1}$ as a union of two smooth parts
\begin{equation*}
  \p \Bigl(\DD^m \times \DD^{n+1}\Bigr) =
  \SS^{m-1} \times \DD^{n+1} \cup  \DD^m \times \SS^n \;.
\end{equation*}
We will denote the first part of the belt sphere by
\begin{align*}
  S_A &:= \Bigl\{\y^- = \0,\, \norm{\y^+}^2 = r^2,\, \norm{\x^-}^2 +
  \norm{\x^+}^2 + \abs{z^\circ}^2 \le r^2 \Bigr\} \cong \SS^{m-1} \times
  \DD^{n+1}\;,
  \intertext{but for now, we will be mostly interested in the second
    part}
  S_B &:= \Bigl\{\y^- = \0,\, \norm{\y^+}^2 \le r^2,\, \norm{\x^-}^2 +
  \norm{\x^+}^2 + \abs{z^\circ}^2 = r^2 \Bigr\} \cong \DD^m \times
  \SS^n\;.
\end{align*}
It lies in the $i$-convex hypersurface $M_B := \bigl\{g_B = r^2 \bigr\}$,
whose complex tangencies are the kernel of the $1$-form $-
\restricted{\bigl(d^c g_B\bigr)}{TM_B}$.
Close to $\{\y^- = \0\}$, we compute
\begin{equation*}
  \begin{split}
    -d^c g_B &= 2\, \sum_{r=1}^k x_r^-\, dy_r^- + 2\, \sum_{s=1}^m
    x_s^+\, dy_s^+ + 2\,\bigl(x^\circ\,dy^\circ - y^\circ \, d
    x^\circ\bigr)
    \intertext{which simplifies on $S_B \subset \{\y^- = \0\}$ further
      to}
    -d^c g_B &= 2\, \sum_{s=1}^m x_s^+\, dy_s^+ +
    2\,\bigl(x^\circ\,dy^\circ - y^\circ \, d x^\circ\bigr) \;.
  \end{split}
\end{equation*}
The submanifold $S_B \cong \DD^m \times \SS^n$ can be foliated by the
$n$-spheres with constant $\y^+$-value.
For every fixed value of $\y^+ = \bbf^+ \in \DD^m_r$ we write the
corresponding leaf as
\begin{equation*}
  L_{\bbf^+}
  =  \Bigl\{(\y^-,\y^+) = (\0, \bbf^+),\, \norm{\x^-}^2
  + \norm{\x^+}^2 + \abs{z^\circ}^2 = r^2 \Bigr\} \;.
\end{equation*}
The restriction of $-d^c g_B$ to each of these spheres is
\begin{equation*}
  2\,\bigl(x^\circ\,dy^\circ - y^\circ \, d x^\circ\bigr)
\end{equation*}
so that $L_{\bbf^+}$ is actually a spherical \LOB in the (strictly)
$i$-convex level set~$M_B$.
The binding~$B_{\bbf^+}$ of the \LOB{}~$L_{\bbf^+}$ is given by the set of
points where $z^\circ$ vanishes, i.e.~$B_{\bbf^+} \cong \SS^{n-2}$; the
pages of the open book are the fibers of the map
\begin{equation*}
  \vartheta\colon   L_{\bbf^+} \setminus B_{\bbf^+} \to \SS^1,\,
  (\z^-, \z^+, z^\circ) \mapsto \frac{z^\circ}{\abs{z^\circ}} \;.
\end{equation*}
In the following sections we will study holomorphic disks that each
have boundary on one of the \LOB{}s~$L_{\bbf^+}$.

\section{The space of holomorphic disks attached to the 
belt sphere}
\label{sec:moduliSpace}

We will now construct the moduli space of pseudoholomorphic disks
needed for the proof of Theorem~\ref{thm: main theorem} and show that
it is a smooth manifold with boundary.
We continue with the setup and notation used in \S\ref{sec: deformed
  subcritical handles} above.

\begin{assumptions}\label{assumptions: choice of J}
  Choose an almost complex structure~$J$ on $\widehat W$ with the
  following properties:
  \begin{itemize}
  \item $J$ is tamed by $\omega$;
  \item $J$ agrees on $\Wmodel$ with the standard complex
    structure~$i$;
  \item the unmodified domain~$\Mreg$ in $\p W$ is $J$-convex, and its
    $J$-complex tangencies agree with $\xi$.
  \end{itemize}
\end{assumptions}

\subsection{The top stratum of the moduli space}

We define $\widetilde\mM\bigl(\widehat W, S_B; J\bigr)$ as the moduli
space of ``parametrized'' curves $\bigl(\bbf^+, u, z_0\bigr)$, where:
\begin{itemize}
\item[(i)] $\bbf^+ \in \DD^m_r$ is a point in the $m$-disk
  parametrizing the \LOB{}s~$L_{\bbf^+} \subset S_B$ described in the
  previous section;
\item[(ii)] $u\colon \bigl(\DD^2, \p \DD^2\bigr) \to \bigl(\widehat W,
  L_{\bbf^+} \setminus B_{\bbf^+})$ is a $J$-holomorphic map which is
  trivial in $\pi_2(\widehat W, L_{\bbf^+})$; and
\item[(iii)] $z_0$ is a marked point in the closed unit disk~$\DD^2$.
\end{itemize}
Additionally we require that
\begin{equation*}
  \restricted{\bigl(\vartheta\circ u\bigr)}{\p \DD^2}\colon
  \SS^1 \to \SS^1
\end{equation*}
is a degree~$1$ map, that is, the boundary of each disk makes one turn
around the binding of the open book.
Since $L_{\bbf^+}$ lies in the strictly convex hypersurface~$M_B$, the
map $\restricted{\bigl(\vartheta\circ u\bigr)}{\p \DD^2}$ is a
diffeomorphism, i.e.~the disk intersects every page of the
\LOB~$L_{\bbf^+}$ precisely once;
see~\cite[Corollary~II.1.11]{NiederkrugerHabilitation}.
It is very easy to deduce from this that the disks in
$\widetilde\mM\bigl(\widehat W, S_B; J\bigr)$ are somewhere injective
near the boundary.
However, since we are not really free to choose the almost complex
structure near the boundary, we need more to achieve transversality.
We say that a $J$-holomorphic disk $u$ is \emph{simple} if it is
somewhere injective in an open dense set of $\DD^2$.
For closed holomorphic curves, simple means not multiply covered, but
for disks the situation is more complicated; see
\cite{LazzariniSimpleDisks}.
However our situation is not the most general one, and it is possible
to adapt the arguments of \cite[Proposition 2.5.1]{McDuffSalamonJHolo}
to prove that all disks in $\widetilde\mM\bigl(\widehat W, S_B;
J\bigr)$ are simple.

\begin{lemma}\label{lemma: our disks are simple}
  Every disk $u \in \widetilde\mM\bigl(\widehat W, S_B; J\bigr)$ is
  simple.
\end{lemma}
\begin{proof}
  We know $u \colon \DD^2 \to W$ is embedded near the boundary.
  Let $X$ denote the set of points $z \in \DD^2$ such that either of
  the following is true:
  \begin{itemize}
  \item[(i)] $D_zu = 0$, or
  \item[(ii)] There exists a different point $z' \in \DD^2$ such that
    $u$ restricted to disjoint neighborhoods of $z$ and $z'$ has an
    isolated intersection $u(z) = u(z')$.
  \end{itemize}
  (Recall that either $u(z) = u(z')$ is an isolated intersection, or
  there exists neighbourhoods of $z$ and $z'$ with the same image.)
  Standard local results plus the fact that $u$ is embedded at $\p
  \DD^2$ tell us that $X$ is a finite set of interior points.
  The image $u(\DD^2 \setminus X)$ is then a smoothly embedded
  $J$-holomorphic submanifold of $W$; in particular, it is a Riemann
  surface $\dot \Sigma$ with connected boundary and finitely many
  punctures.
  The inclusion of $\dot \Sigma$ into $W$ is then a $J$-holomorphic
  embedding, and it extends over the punctures to a $J$-holomorphic
  map $v \colon \Sigma \to W$, which is not necessarily an embedding
  but has only finitely many critical points and self-intersections.
  At this point we don't know the topology of $\Sigma$, except that it
  has connected boundary.
  But the original map $u$, restricted to $\DD^2 \setminus X$, defines
  a holomorphic map to $\dot \Sigma$, which then extends by removal of
  singularities to a holomorphic map $\varphi \colon \DD^2 \to \Sigma$
  such that $u = v \circ \varphi$.
  Given the properties of $u$ at the boundary, $\varphi$ must restrict
  to a diffeomorphism $\p \DD^2 \to \p \Sigma$, and it maps interior
  to interior.
  So it has degree one, and is therefore biholomorphic.
\end{proof}

Lemma~\ref{lemma: our disks are simple} will allow us to use the
following transversality result.

\begin{proposition}\label{prop: smoothness of moduli space}
  Let $(W,J)$ be an almost complex manifold, and let $\DD^m_\epsilon
  \times L \subset W$ be a submanifold for which every slice $L_\x :=
  \{\x\} \times L$ is a totally real submanifold.
  For generic choices of $J$ satisfying Assumptions~\ref{assumptions:
    choice of J}, the following holds.
  Suppose $u_0\colon \bigl(\DD^2, \p \DD^2\bigr) \to \bigl(W,
  L_\0\bigr)$ is any $J$-holomorphic map such that
  \begin{itemize}
  \item the interior points of $u_0$ do not touch the boundary of $W$;
  \item the boundary of $u_0$ lies in the interior of $L_\0$;
  \item the disk~$u_0$ is \emph{simple}.
  \end{itemize}
  Let $\widetilde\mM$ be the space of all $J$-holomorphic maps
  \begin{equation*}
    u\colon \bigl(\DD^2, \p \DD^2\bigr) \to \bigl(W, L_\x\bigr)
  \end{equation*}
  for all $\x \in \DD^m_\epsilon$.
  Then the space of solutions in $\widetilde\mM$ close to $u_0$ forms
  a smooth ball that has $u_0$ as its center and whose dimension is
  \begin{equation*}
    \dim \mM = \frac{1}{2} \dim W + \mu\bigl(u_0^*T W, u_0^*T L_\0\bigr) + m \;,
  \end{equation*}
  where $\mu\bigl(u_0^*T W, u_0^*T L_\0\bigr)$ denotes the Maslov
  index of the disk~$u_0$.
\end{proposition}
\begin{proof}
  The result is standard if $m=0$, in which case $\frac{1}{2} \dim W +
  \mu(u_0^*TW,u_0^*TL_0)$ is the Fredholm index of the linearized
  Cauchy-Riemann operator on a suitable Banach space of sections of
  $u_0^*TW$ with totally real boundary condition; see
  \cite[Section~3.2]{McDuffSalamonJHolo}.
  For $m > 0$, the linearized problem is the same as that of the $m=0$
  case, but with an extra $m$-dimensional space of smooth sections
  added to the domain in order to allow for the moving boundary
  condition, cf.~\cite[\S 4.5]{Wendl_thesis}.
  Thus the Fredholm index becomes larger by~$m$.
  Given the corresponding enlargement of the nonlinear configuration
  space, the proof of transversality for generic $J$ works as in the
  standard case by defining a suitable universal moduli space and
  applying the Sard-Smale theorem, see
  e.g.~\cite[Chapter~3]{McDuffSalamonJHolo}.
\end{proof}

Recall from Section~\ref{sec: deformed subcritical handles} that the
family of \LOB{}s is parametrized by a disk $\DD^{m}_r$ of some fixed
radius $r \ll 1$.
We define
\begin{equation*}
  \widetilde \mMint\bigl(\widehat W, S_B; J\bigr) =
  \bigl\{ \bigl(\bbf^+, u, z_0 \bigr) \in
  \widetilde\mM\bigl(\widehat W, S_B; J\bigr) \bigm|\, \norm{\bbf^+}
  < r \bigr\} \;.
\end{equation*}
Since $g_B$ is plurisubharmonic and $g_A$ is weakly plurisubharmonic,
this subspace consists of $J$-holomorphic disks that map the interior
of $\DD^2$ to the interior of $\widehat W$; see also
Proposition~\ref{prop: only Bishop disks touch boundary at interior
  points}.

\begin{corollary}\label{cor: boh}
  The subspace $\widetilde \mMint\bigl(\widehat W, S_B; J\bigr)$ of
  the parametrized moduli space is a smooth manifold with boundary,
  and its dimension is $2n-k+3$, where $k$ is the index of the
  surgery.
  Its boundary consists of triples $\bigl( \bbf^+, u, z_0 \bigr)$ with
  $z_0 \in \p \DD^2$.
\end{corollary}

Note that triples $\bigl( \bbf^+, u, z_0 \bigr)$ with $\norm{\bbf^+} =
r$ do not belong to $\widetilde \mMint\bigl(\widehat W, S_B; J\bigr)$,
and therefore are not points of its boundary.

\begin{proof}
  By definition the elements of $\widetilde \mMint\bigl(\widehat W,
  S_B; J\bigr)$ satisfy the hypotheses of Proposition~\ref{prop:
    smoothness of moduli space}, and therefore every $J$-holomorphic
  disk $u \in \widetilde \mMint\bigl(\widehat W, S_B; J\bigr)$ has an
  open neighbourhood which is diffeomorphic to a ball of dimension
  $\frac 12 \dim \widehat{W} + \mu \bigl( u^*T\widehat W, u^*
  TL_{\bbf^+}\bigr) +m +2$ --- the presence of the marked point adds
  $2$ to the index.
  A simple computation shows that $\mu \bigl( u^*T\widehat W, u^*
  TL_{\bbf^+}\bigr) =2$ for all $(\bbf^+, u, z_0)$, so that
  $\widetilde \mMint\bigl(\widehat W, S_B; J\bigr)$ is a smooth
  manifold with boundary of dimension
  \begin{equation*}
    \dim \widetilde \mMint\bigl(\widehat W, S_B; J\bigr) = n+m+4 \;.
  \end{equation*}
  (The boundary points are those with the mark point in $\partial
  \DD^2$.)
  Since $m=n-k-1$, we obtain the desired formula for the dimension.
\end{proof}

In the next subsection we will analyze what happens when
$\norm{\bbf^+} =r$, and we will also show that
$\widetilde\mM\bigl(\widehat W, S_B; J\bigr)$ is non-empty.
To consider geometric disks instead of parametrized ones, we divide
$\widetilde \mM\bigl(\widehat W, S_B; J\bigr)$ by the group of
biholomorphic reparametrizations of $\DD^2\subset \CC$.
We define the moduli space of ``unparametrized'' curves:
\begin{equation*}
  \mM \bigl(\widehat W, S_B; J\bigr) =
  \widetilde \mM\bigl(\widehat W, S_B; J\bigr) / \sim \;,
\end{equation*}
where $\bigl(\bbf^+, u, z_0\bigr) \sim \bigl(\widetilde \bbf^+,
\widetilde u, \widetilde z_0\bigr)$ if and only if $\bbf^+ =
\widetilde \bbf^+$ and there exists a transformation
$\varphi \in \Aut(\DD^2)$ with $u = \widetilde u\circ
\varphi^{-1}$ and $z_0 = \varphi(\widetilde z_0)$.
The action of the reparametrization group $\Aut(\DD^2)$ preserves
$\widetilde \mMint\bigl(\widehat W, S_B; J\bigr)$; we denote its 
quotient by $\mMint\bigl(\widehat W, S_B; J\bigr)$.

\begin{proposition}\label{prop: topology of interior moduli space}
  The subspace $\mMint\bigl(\widehat W, S_B; J\bigr)$ of the moduli
  space is a smooth manifold with boundary of dimension $2n-k$.
\end{proposition}
\begin{proof}
  The map $\vartheta \colon S_B \setminus \{ z^\circ =0 \} \to
  \SS^1$ is globally defined for all \LOB{}s in the family.
  Therefore we can define the subset
  \begin{equation*}
    \widetilde \mM_0\bigl(\widehat W, S_B; J\bigr) \subset \widetilde \mMint
    \bigl(\widehat W, S_B; J\bigr)
  \end{equation*}
  consisting of triples $\bigl( \bbf^+, u, z_0 \bigr)$ such that
  \begin{equation}\label{eq: slice}
    \vartheta\bigl(u(z)\bigr) = 
    \begin{cases}
      1 & \text{if $z = 1$,} \\
      i & \text{if $z = i$,} \\
      -1 & \text{if $z = -1$}. \\
    \end{cases}
  \end{equation}
  We know that $\widetilde \mM_0\bigl(\widehat W, S_B; J\bigr)$ is a
  submanifold of $\widetilde \mMint\bigl(\widehat W, S_B; J\bigr)$
  because $\vartheta \circ \restricted{u}{\p\DD^2}$ is a
  diffeomorphism and the biholomorphism group of the disk is triply
  transitive on $\p\DD^2$.
  Then the subset $\widetilde \mM_0\bigl(\widehat W, S_B; J\bigr)$
  provides a global slice for the action of $\Aut(\DD^2)$ on
  $\widetilde \mMint\bigl(\widehat W, S_B; J\bigr)$.
\end{proof}

\subsection{The Bishop disks}
\label{sec: bishop disks at the boundary}

In this section, we want to study a certain class of disks in $\mM
\bigl(\widehat W, S_B; J\bigr)$ that lie entirely in the model
neighborhood~$\Wmodel$ and that can be described explicitly.
A \defin{Bishop disk} is a disk that we obtain by intersecting a
$z^\circ$-plane in $\CC^n$ with constant $(\z^-, \z^+)$-coordinates
with the model neighborhood~$\Wmodel$.
A possible way to parametrize it is as a map
\begin{equation*}
  u\colon \bigl(\DD^2, \p\DD^2\bigr) \to \bigl(W, L_{\bbf^+}\bigr) \;,
\end{equation*}
with constant coordinates $(\y^-,\y^+) = (\0, \bbf^+)$, constant
$\x^-$ and $\x^+$-coordinates, so we write
\begin{equation*}
  u(z) =  \bigl(\x^-; \x^+ + i\,\bbf^+; Cz\bigr)\;,
\end{equation*}
where $C = \sqrt{r^2 - \norm{\x^-}^2 - \norm{\x^+}^2}$.
The Bishop disks are the buds from which the moduli space will grow,
and it is therefore important to establish that they are Fredholm
regular, meaning that their linearized Cauchy-Riemann operators are
surjective.
This is ensured by the following ``automatic'' transversality lemma
(see~\cite[Section~III.1.3]{NiederkrugerHabilitation}).

\begin{lemma}\label{lemma: Bishop disks are regular}
  Let $u \colon \DD^2 \to \widehat W$ be a Bishop disk with image in
  $\Wmodel$ and boundary mapped to the \LOB~$L_{\bbf^+}$.
  Then its linearized Cauchy-Riemann operator~$D_{(\bbf^+, u)}$,
  defined on suitable Banach space completions with totally real
  boundary condition determined by the \LOB~$L_{\bbf^+}$, is
  surjective.
\end{lemma}
\begin{corollary}\label{cor: Bishop disks are regular}
  The triples $(\bbf^+, u, z_0)$ where $u$ is a Bishop disk with image
  in $\Wmodel$ are regular points of the moduli space $\widetilde \mM
  \bigl(\widehat W, S_B; J\bigr)$.
\end{corollary}
\begin{proof}
  The relevant linearized operator is the same as $D_{(\bbf^+, u)}$ in
  Lemma~\ref{lemma: Bishop disks are regular}, except that the moving
  boundary condition satisfied by $J$-holomorphic maps in $\widetilde
  \mM\bigl(\widehat W, S_B; J\bigr)$ means that this domain must be
  enlarged by some finite-dimensional space of smooth sections,
  allowing the boundary to move to different \LOB{}s in the family
  (see Appendix~\ref{sec: orientability of the moduli space} for more
  details).
  The target of the operator remains the same, so surjectivity of
  $D_{(\bbf^+, u)}$ in Lemma~\ref{lemma: Bishop disks are regular}
  immediately implies surjectivity on the enlarged domain.
\end{proof}

The rest of this subsection will be concerned with the proof that the
Bishop disks are the only holomorphic disks in $\Wmodel$.

\begin{proposition}\label{prop: only Bishop disks touch boundary at
    interior points}
  If a holomorphic disk
  \begin{equation*}
    u\colon \bigl(\DD^2, \p\DD^2\bigr) \to \bigl(\widehat W, L_{\bbf^+}\bigr)
  \end{equation*}
  touches the boundary of $\widehat W$ at an interior point of
  $\DD^2$, then either it is constant or it is a multiple cover of a
  Bishop disk that is completely contained in $S_A \subset M_A \cap
  \beltSphere{n+m}$.
\end{proposition}
\begin{proof}
  Let $z_0 \in \mathring \DD^2$ be a point in the interior of the disk
  at which $u$ touches $M_A$, $M_B$, or $\Mreg$.
  We will obtain the desired statement by using the maximum principle;
  we only need to be a bit more careful compared with the standard
  situation, because the boundary of $\widehat W$ is defined piecewise
  as a union of level sets of different plurisubharmonic functions.
  Assume first that $u(z_0)$ touches $M_B$.
  The function~$g_B$ is not defined on the whole symplectic filling,
  but we may nonetheless assume that $g_B$ exists on a small
  neighborhood of $u(z_0)$, hence we find an open subset $U \subset
  \mathring \DD^2$ containing $z_0$ such that
  \begin{equation*}
    \restricted{\bigl(g_B\circ u\bigr)}{U}\colon U \to \RR
  \end{equation*}
  is a plurisubharmonic function having a maximum at $z_0$.
  It follows from the maximum principle that $\restricted{g_B\circ
    u}{U}$ is constant, and due to strong convexity it even follows
  that the holomorphic map $\restricted{u}{U}$ itself must be
  constant.
  This implies that the open set~$U$ chosen above can in fact be
  extended to the whole disk, and $u$ will be a constant disk.
  Note that this argument also remains valid if $u(z_0)$ lies in the
  edge where $M_A$ and $M_B$ meet.
  The disk lies in the model locally in the domain with $g_B \le r^2$,
  and thus $\restricted{g_B\circ u}{U}$ still has a local maximum at
  $z_0 \in \mathring\DD^2$, as used previously.
  Similarly, the argument can be used verbatim for disks that touch
  $\Mreg$, and this implies in fact that there are no disks at all
  touching $\Mreg$ at interior points, because a constant disk must
  lie in $L_{\bbf^+} \subset \beltSphere{n+m}$, which is disjoint from
  $\Mreg$.
  Let us now assume that the disk~$u$ touches the hypersurface~$M_A$
  at $z_0$.
  Again, we find an open subset~$U \subset \mathring \DD^2$ containing
  $z_0$ for which
  \begin{equation*}
    \restricted{\bigl(g_A\circ u\bigr)}{U}\colon U \to \RR
  \end{equation*}
  is defined and has a maximum at $z_0$.
  By weak plurisubharmonicity, this function must be constant.
  Now it is easy to see that we can choose $U$ to be the whole
  disk~$\DD^2$, because by continuity, the image of every point~$z \in
  \overline{U}$ lies in $\p\widehat W$.
  If $z$ is an interior point of the disk, and if $u(z)$ is an
  interior point of $M_A$, i.e.~it does not lie in $M_A\cap M_B$, then
  we can extend $U$ to a larger open domain that contains $z$ in its
  interior.
  If $z$ is an interior point but $u(z)$ \emph{does} lie in $M_A\cap
  M_B$, then we know by the first part of the proof that $u$ must be a
  constant map.
  In both cases the whole disk lies in $M_A$.
  It remains to see that a nonconstant holomorphic disk lying in
  $M_A$ must be a Bishop disk (or a multiple cover).
  We know that all coordinate functions are harmonic, and hence each
  of them must attain both its maximum and its minimum at a point on
  the boundary of the disk.
  The boundary of $u$ lies in $L_{\bbf^+} \subset \bigl\{\y^- =
  \0\bigr\}$, and hence it follows that all of the $\y^-$-coordinates
  vanish on the disk.
  From the Cauchy-Riemann equation, we then see that the
  $\x^-$-coordinates of the disk will be constant.
  Similarly, the $\y^+$-coordinates of the disk must all be equal to
  $\bbf^+$, because $L_{\bbf^+} \subset \bigl\{\y^+ = \bbf^+\bigr\}$,
  and again by the Cauchy-Riemann equation also the $\x^+$-coordinates
  will be constant.
  The only nonconstant coordinate functions of the disk are the
  $z^\circ$-coordinate, and they span a round disk.
\end{proof}

Recall that
\begin{equation*}
  B_{\bbf^+} = L_{\bbf^+} \cap  \{z^\circ = 0\}
\end{equation*}
is the binding of the \LOB~$L_{\bbf^+}$.

\begin{proposition}\label{prop: bishop disks close to binding}
  There exists an open subset $V \subset \Wmodel$, containing
  $B_{\bbf^+}$ for every $\bbf^+ \in \DD^m_r$, such that every
  holomorphic disk
  \begin{equation*}
    u\colon \bigl(\DD^2, \p\DD^2\bigr) \to \bigl(\widehat W,
    L_{\bbf^+}\bigr)
  \end{equation*}
  in $\widetilde \mM\bigl(\widehat W, S_B; J\bigr)$ intersecting $V$
  must be a Bishop disk up to reparametrization.
\end{proposition}
\begin{proof}
  Note that
  \begin{equation*}
    h(\z^-,\z^+, z^\circ) =  \norm{\x^-}^2 - \frac{1}{2} \, \norm{\y^-}^2
    + \norm{\x^+}^2
  \end{equation*}
  is a weakly plurisubharmonic function on $\Wmodel$.
  Its value on the binding~$B_{\bbf^+}$ is $r^2$, and it decreases
  along the \LOB.
  If we choose a sufficiently small $\epsilon>0$, we can make sure
  that $V := h^{-1}\bigl((r^2-\epsilon,+ \infty)\bigr) \cap \Wmodel$
  is an open neighborhood of $B_{\bbf^+}$ with $\overline{V} \subset
  \Wmodel$.
  It follows in fact from $g_B \le r^2$ and $h > r^2 - \epsilon$ that
  \begin{equation*}
    g_B(\z^-,\z^+, z^\circ)-h(\z^-,\z^+, z^\circ)= \frac{1}{2}\, \norm{\y^-}^2
    + \psi\bigl(\norm{\y^-}\bigr)\cdot \norm{\y^+}^2
    + \abs{z^\circ}^2 < \epsilon \;,
  \end{equation*}
  so that both the $\y^-$ and the $z^\circ$-coordinates are small in
  $V$, and in particular we can assume that $\psi=0$ on $V$.
  On the other hand,
  \begin{equation*}
    \norm{\x^-}^2 + \norm{\x^+}^2 > r^2 - \epsilon
    +  \frac{1}{2} \, \norm{\y^-}^2 \ge r^2 - \epsilon
  \end{equation*}
  implies that every point in $V$ lies in an arbitrarily small
  neighborhood of $S_B$.
  Let now $u\colon \bigl(\DD^2, \p\DD^2\bigr) \to \bigl(\widehat W,
  L_{\bbf^+}\bigr)$ be a holomorphic disk whose image intersects $V$.
  Assume that $h\circ u$ is not constant: then we can choose by Sard's
  theorem a slightly smaller number $\epsilon' < \epsilon$ for which
  $r^2 - \epsilon'$ will be a regular value of $h\circ u$, so that the
  subdomain
  \begin{equation*}
    G := \bigl\{z \in \DD^2\bigm|\, (h\circ u)(z) \ge r^2-\epsilon' \}
  \end{equation*}
  is compact and has piecewise smooth boundary, which we denote by
  \begin{equation*}
    \p G = \p_+G \cup  \p_-G \;,
  \end{equation*}
  where $\p_+ G = G \cap \p\DD^2$ lies in the boundary of the unit
  disk, and $\p_- G$ lies in the interior of the unit disk.
  Denote the restriction
  \begin{equation*}
    \restricted{u}{G}\colon G \to \Wmodel
  \end{equation*}
  by $u_G$.
  By the maximum principle, it follows that the maximum of $h\circ
  u_G$ on each component of $G$ must lie on the boundary of that
  component.
  Clearly then the boundary of every component of $G$ must intersect
  $\p_+G$, because otherwise $h\circ u_G$ would have an interior
  maximum on that component, so it would be equal to $r^2-\epsilon'$,
  but this contradicts the assumption that $r^2-\epsilon'$ is a
  regular value.
  It follows then that every component of $G$ must intersect
  $\p\DD^2$, and since $h\circ u_G$ is minimal along $\p_- G$, the
  maximum of $h\circ u_G$ must lie at a point $z_0 \in \p_+G \subset
  \p\DD^2$.
  By the boundary point lemma, a version of the maximum principle at
  the boundary (see for example
  \cite[Theorem~II.1.3]{NiederkrugerHabilitation}), the derivative of
  $h\circ u_G$ at $z_0$ in the outward radial direction must be
  strictly positive.
  We choose polar coordinates $(r, \varphi)$ on $\DD^2$. Using the
  fact that $u$ is $J$-holomorphic, we can write
  \begin{equation*}
    \begin{split}
      \partial_r\bigl(h\circ u\bigr) &= dh\bigl( Du
      \cdot \partial_r\bigr)
      = dh \bigl(Du \cdot (-i\cdot\partial_\varphi) \bigr) \\
      &= - dh\bigl(i \cdot Du \cdot \partial_\varphi\bigr) = - d^ch
      \bigl(Du \cdot \partial_\varphi\bigr) \;,
    \end{split}
  \end{equation*}
  but note that
  \begin{equation*}
    - d^c h =  \sum_{r=1}^k \Bigl(2 x_r^-\, dy_r^-
    + \ y_r^- \, dx_r^-\Bigr) + 2\,\sum_{s=1}^m  x_s^+\, dy_s^+ \;.
  \end{equation*}
  We obtain $- d^ch \cdot Du \cdot \partial_\varphi = 0$ along the
  whole boundary of the disk, because the boundary of $u$ lies in the
  \LOB~$L_{\bbf^+}$, which is a subset of $\{\y^- = \0, \y^+ =
  \bbf^+\}$.
  It follows that $\partial_r\bigl(h\circ u\bigr) = - d^ch \bigl(Du
  \cdot \partial_\varphi\bigr)$ vanishes at $z_0$, and by the boundary
  point lemma, the disk must be contained in one of the level sets of
  $h$, so in particular it lies in $V \subset \Wmodel$.
  The rest of the statement follows from standard arguments.
  All of the coordinate functions on $\Wmodel$ are harmonic, hence
  they must attain their maxima and minima on the boundary of the
  disk.
  Since $\y^- = \0$ along $\p\DD^2$, the $\y^-$-coordinates of $u$ are
  zero on the whole disk, and using the Cauchy-Riemann equation, we
  see that the $\x^-$-coordinates must be constant on the disk.
  Similar arguments work for $\y^+$ and $\x^+$, and we finally
  conclude that $u$ must be a Bishop disk.
\end{proof}

\begin{proposition}\label{prop: bishop disks on Lobs in MA}
  Let $L_{\bbf^+}$ be a \LOB that lies in the hypersurface~$M_A$,
  i.e.~$\bbf^+ \in \DD^m_r$ has been chosen such that $\norm{\bbf^+} =
  r$.
  Then up to parametrization, every holomorphic disk~$u$ in
  $\widetilde \mM\bigl(\widehat W, S_B; J\bigr)$ whose boundary lies
  in $ L_{\bbf^+}$ is a Bishop disk.
\end{proposition}
\begin{proof}
  Note that $L_{\bbf^+}$ lies in the level set of the weakly
  plurisubharmonic function $g_A\colon \widehat W \to \RR$.
  It suffices to prove that the image of $u$ has to lie entirely in
  $M_A \subset \Wmodel$, as this already implies the desired statement
  by Proposition~\ref{prop: only Bishop disks touch boundary at
    interior points}.
  Since the whole boundary $u\bigl(\p\DD^2\bigr)$ lies in $M_A$, we
  can find a closed annulus $G \subset \DD^2$ having $\p\DD^2$ as one
  of its boundary components such that
  \begin{equation*}
    \restricted{\bigl(g_A\circ u\bigr)}{G}\colon G \to \RR
  \end{equation*}
  is defined and everywhere weakly plurisubharmonic, and it takes its
  maximum along $\p\DD^2 \subset G$.
  Assume first that the disk~$u$ is tangent to $M_A$ at one of its
  boundary points.
  We can apply the boundary point lemma around this point (see again
  \cite[Theorem~II.1.3]{NiederkrugerHabilitation}) to deduce that
  $\restricted{\bigl(g_A\circ u\bigr)}{G}$ has to be constant on all
  of $G$.
  In particular this implies that $u(G)$ lies in $M_A$, and $u$
  touches $M_A$ also with one of its interior points.
  Proposition~\ref{prop: only Bishop disks touch boundary at interior
    points} then implies that $u$ is either constant or one of the
  Bishop disks.
  Conversely suppose that $u$ is everywhere transverse to $M_A$,
  meaning that $\partial_r \bigl(g_A\circ u\bigr) (z)$ is
  \emph{strictly} positive for every $z \in \p\DD^2$.
  The restriction $\restricted{u}{G}$ is a $J$-holomorphic map whose
  image lies in $\Wmodel$; moreover, $g_A\circ u \equiv r^2$ on
  $\p\DD^2$ and $g_A\circ u < r^2$ on the inner boundary of $G$.
  Introduce on $G$ the polar coordinates $z= \rho e^{i\varphi}$.
  Note that along $\p\DD^2$, all of the $\y^+$-coordinates are
  constant in the $\varphi$-direction, because the boundary of the
  disk lies in the \LOB~$L_{\bbf^+}$.
  Multiplying the complex coordinates~$\z^+$ by a suitable
  $\SO(m)$-matrix (the standard complex structure~$i$ and the
  functions $g_A$, $g_B$ are invariant under such a multiplication),
  we may assume that $\bbf^+ = (r,0,\dotsc,0)$.
  It follows that the $y_1^+$-coordinate of $\restricted{u}{G}$ has
  its maximum on $\p\DD^2$.
  Note now that the $x_1^+$-coordinate of $\restricted{u}{\SS^1}$ is
  bounded, and hence it necessarily must take a maximum at some point
  $e^{i\varphi_0} \in \SS^1 = \p\DD^2$, so that
  \begin{equation*}
    \restricted{\frac{d}{d\varphi}}{\varphi = \varphi_0}
    x_1^+\Bigl(u\bigl(e^{i\varphi}\bigr)\Bigr) =  0\;.
  \end{equation*}
  Again, we can use complex multiplication to see
  $i\cdot \partial_\rho = \partial_\varphi$, hence
  \begin{equation*}
    dy_1^+\bigl(Du \cdot \partial_\rho\bigr) =
    dy_1^+\bigl(Du \cdot (-i\cdot\partial_\varphi)\bigr)
    = - dy_1^+\bigl(i\cdot Du\cdot\partial_\varphi\bigr)
    = - dx_1^+ \bigl(Du\cdot\partial_\varphi\bigr) \;,
  \end{equation*}
  and in particular the radial derivative of $y_1^+$ vanishes at
  $e^{i\varphi_0}$, so that by the boundary point lemma, $y_1^+$ must
  be constant on all of $G$.
  Using the fact that $r^2 = \abs{y_1^+}^2 \le g_A(\z^-,\z^+, z^\circ)
  \le r^2$ everywhere on $G$, we deduce that all of $u(G)$ lies in
  $M_A$.
  In particular, $u$ touches $M_A$ at an interior point, which allows
  us to conclude the proof by applying Proposition~\ref{prop: only
    Bishop disks touch boundary at interior points}.
\end{proof}

We end this subsection with a description of the global topology of
the moduli spaces $\widetilde \mM\bigl(\widehat W, S_B; J\bigr)$ and
$\mM\bigl(\widehat W, S_B; J\bigr)$.

\begin{proposition}\label{prop: topology of the uncompactified moduli
    space}
  The parametrized moduli space $\widetilde \mM\bigl(\widehat W, S_B;
  J\bigr)$ is a smooth $(2n-k+3)$-dimensional manifold with boundary
  and corners.
  Its boundary has two smooth strata, one corresponding to elements
  $(\bbf^+, u, z_0)$ with $\norm{\bbf^+} =r$, and the other
  corresponding to elements $(\bbf^+, u, z_0)$ with $\abs{z_0}=1$.
  The moduli space $ \mM\bigl(\widehat W, S_B; J\bigr)$ is a smooth
  $(2n-k)$-dimensional manifold with boundary and corners, which
  decomposes as a product
  \begin{equation*}
    \mM\bigl(\widehat W, S_B; J\bigr) = \Sigma \times \DD^2\;,
  \end{equation*}
  where $\Sigma$ is a (non-compact) manifold with boundary.
\end{proposition}

\begin{proof}
  Let $(\bbf^+,u,z_0)$ be an element of $\widetilde
  \mMint\bigl(\widehat W, S_B; J\bigr)$.
  Since $\widetilde \mMint\bigl(\widehat W, S_B; J\bigr)$ is open in
  $\widetilde \mM\bigl(\widehat W, S_B; J\bigr)$, it follows from
  Proposition~\ref{prop: only Bishop disks touch boundary at interior
    points} that the image of $u$ does not touch $\p \widehat W$ with
  any interior point if $\norm{\bbf^+} <r$, and that $(\bbf^+,u,z_0)$
  has a neighborhood in $\widetilde \mM\bigl(\widehat W, S_B; J\bigr)$
  which is diffeomorphic to an open ball $\mathring \DD^{2n-k+3}$ if
  $\abs{z_0}<1$, or to a half-ball $\mathring \DD^{2n-k+1}\times
  \bigl\{z \in \CC\bigm|\, \ImaginaryPart z \ge 0 \bigr\}$ if
  $\abs{z_0}=1$.
  Now we consider the elements $(\bbf^+,u,z_0)$ of $\widetilde
  \mM\bigl(\widehat W, S_B; J\bigr)$ such that the image of $u$
  touches $\p\widehat W$ at an interior point.
  Again by Proposition~\ref{prop: only Bishop disks touch boundary at
    interior points} this implies $\norm{\bbf^+}=r$.
  We know by Proposition~\ref{prop: bishop disks on Lobs in MA} that
  $u$ will be a Bishop disk up to reparametrization. Since Bishop
  disks are regular by Corollary~\ref{cor: Bishop disks are regular},
  a neighborhood of $(\bbf^+, u, z)$ in $\widetilde \mM\bigl(\widehat
  W, S_B; J\bigr)$ looks like the neighborhood of a boundary point of
  an $(n+m+2)$-manifold.
  A priori the boundary of the unparametrized moduli space may contain
  one more stratum consisting of triples $(\bbf^+, u, z_0)$ such that
  the image of $u$ touches the binding $B_{\bbf^+}$ of the
  \LOB~$L_{\bbf^+}$ at a boundary point.
  However, in this case, the image of $u$ would have to intersect the
  neighborhood~$V$ from Proposition~\ref{prop: bishop disks close to
    binding}, and therefore $u$ would be a Bishop disk.
  Since Bishop disks which intersect the binding are constant and, by
  definition, do not belong to $\widetilde \mM\bigl(\widehat W, S_B;
  J\bigr)$, the possibility that the image of $u$ touches $B_{\bbf^+}$
  cannot occur.
  (We will see later that those constant disks must be added to the
  compactification of the moduli space.)

  By definition, the space $\widetilde \mM\bigl(\widehat W, S_B;
  J\bigr)$ of parametrized disks with a marked point is a trivial disk
  bundle.
  The moduli space~$\mM\bigl(\widehat W, S_B; J\bigr)$ is then a disk
  bundle over the moduli space of the same holomorphic curves without
  the marked point, and the projection map is simply the map that
  forgets the marked point.
  In our case though, it is even true that $\mM\bigl(\widehat W, S_B;
  J\bigr)$ is a \emph{trivial} disk bundle.
  Recall that the map $\vartheta\colon S_B \setminus \{z^\circ = 0\}
  \to \SS^1$ is globally defined for all \LOB{}s in the family.
  Hence every equivalence class~$[\bbf^+, u, z]$ in $\mM\bigl(\widehat
  W, S_B; J\bigr)$ has a unique representative $(\bbf^+, u_0, z_0)$,
  defined by fixing a parametrization of $u$ such that
  \begin{equation*}
    \vartheta\bigl(u(z)\bigr) = 
    \begin{cases}
      1 & \text{if $z = 1$,} \\
      i & \text{if $z = i$,} \\
      -1 & \text{if $z = -1$,} \\
    \end{cases}
  \end{equation*}
  as we did in the proof of Proposition~\ref{prop: topology of
    interior moduli space}.
  This choice of parametrization gives a global slice for the action
  of $\Aut(\DD^2)$ on $\widetilde \mM\bigl(\widehat W,
  S_B; J\bigr)$, which identifies $\mM\bigl(\widehat W, S_B; J\bigr)$
  with a subset of $\widetilde \mM\bigl(\widehat W, S_B; J\bigr)$.
  Then $\mM\bigl(\widehat W, S_B; J\bigr)$ is a trivial disk bundle
  because the same is true of $\widetilde \mM\bigl(\widehat W, S_B;
  J\bigr)$.
\end{proof}

\subsection{Topology of the compactified moduli space}
\label{sec: topology moduli space}

In the previous sections, we introduced the moduli space we want to
use, and we showed that all the disks intersecting certain domains of
the model neighborhood~$\Wmodel$ must be Bishop disks.
Our aim in this section is to study the topology of the natural
compactification of that moduli space.
The compactification of $\mM\bigl(\widehat W, S_B; J\bigr)$ involves
two phenomena:
(1)~Gromov convergence to stable nodal holomorphic disks (see
e.g.~\cite{Frauenfelder_disks,FrauenfelderZehmisch_disks}), and
(2)~degeneration to constant maps in the binding of a \LOB.
In order to accommodate the latter without losing the extra
disk-bundle structure provided by the marked point, we shall (as in
Proposition~\ref{prop: topology of interior moduli space}) replace
$\mM\bigl(\widehat W, S_B; J\bigr)$ by the space $\widetilde
\mM_0\bigl(\widehat W, S_B; J\bigr)$ of parametrized curves that
satisfy the condition~\eqref{eq: slice}.
This introduces a hint of extra book-keeping into the following
statement, but the reader should keep in mind that the space we are
actually interested in is always $\mM\bigl(\widehat W, S_B; J\bigr)$.

\begin{proposition}\label{prop: Gromov compactness}
  Any sequence $(\bbf^+_j, u_j, z_j) \in \widetilde \mM\bigl(\widehat
  W, S_B; J\bigr)$ satisfying the condition~\eqref{eq: slice} has a
  subsequence that converges to a unique configuration of one of the
  following types:
  \begin{enumerate}
  \item An element of the moduli space $(\bbf^+_\infty, u_\infty,
    z_\infty) \in \widetilde \mM\bigl(\widehat W, S_B; J\bigr)$, still
    satisfying~\eqref{eq: slice};
  \item A triple $(\bbf^+_\infty, p_\infty, z_\infty)$, where
    $\bbf^+_\infty \in \DD^m$, $p_\infty$ represents the constant map
    at the point $p_\infty \in B_{\bbf^+_\infty}$ and $z_\infty \in
    \DD^2$; or
  \item A triple $(\bbf^+_\infty, \mathfrak{t}_\infty, z_\infty)$,
    where $\bbf^+_\infty \in \DD^m$, $\mathfrak{t}_\infty$ is a stable
    nodal holomorphic disk with boundary on $L_{\bbf^+_\infty}
    \subset\beltSphere{n+m}$, consisting of a single nonconstant disk
    with a tree of sphere bubbles attached, and $z_\infty$ is a marked
    point on the domain of $\mathbf{t}_\infty$.
  \end{enumerate}
  Convergence in cases~(1) and (2) is in the $C^\infty$-topology, and
  in case (3) it is in the sense of Gromov.
  If $W$ is symplectically aspherical, then the third case does not
  occur.
\end{proposition}

\begin{proof}
  Since the parameters~$\bbf^+_j$ belong to the closed ball~$\DD^m_r$,
  we can extract a first subsequence from $(\bbf^+_j, u_j, z_j)$ for
  which the parameters converge to some limit $\bbf^+_\infty \in
  \DD^m_r$.
  For simplicity, we still denote this subsequence by $(\bbf^+_j, u_j,
  z_j)$.
  The usual statement of Gromov's compactness theorem for holomorphic
  disks (see \cite[\S 4]{Frauenfelder_disks}) applies to sequences of
  \emph{unparametrized} curves with fixed numbers of interior marked
  points and/or boundary marked points.
  Thus in order to apply the theorem to $(\bbf^+_j, u_j, z_j) \in
  \widetilde \mM\bigl(\widehat W, S_B; J\bigr)$, it will be convenient
  to observe that parametrized curves satisfying \eqref{eq: slice} can
  be identified in a canonical way with unparametrized stable nodal
  $J$-holomorphic disks carrying one interior marked point
  (corresponding to $z_j$) and three extra boundary marked points
  (corresponding to the points $1,i,-1$), where the latter are
  required to satisfy incidence conditions under the evaluation map.
  In this picture, \emph{smooth} (i.e.~non-nodal) unparametrized
  curves with extra boundary marked points correspond to triples
  $(\bbf^+_j,u_j,z_j)$ with $z_j \in \DD^2 \setminus \p\DD^2$, and
  triples with $z_j \in \p\DD^2$ are identified with \emph{nodal}
  curves that consist of a nonconstant disk $u_j$ attached by a node
  at $z_j$ to a single constant (``ghost'') disk on which the interior
  marked point lives.
  With this identification understood, suppose the maps~$u_j$ have
  images bounded away from the binding $B_{\bbf^+_\infty}$.
  Then after taking a subsequence, we can assume by Gromov compactness
  that the corresponding sequence of unparametrized stable curves with
  extra boundary marked points converges in the Gromov topology to a
  smooth or nodal $J$-holomorphic disk.
  Note that each of the unparametrized curves has a unique
  parametrization for which the (ordered) set of boundary marked
  points is $(1,i,-1)$, thus if the Gromov limit is smooth, then this
  means $z_j$ converges to an interior point of $\DD^2$ and $u_j$
  converges in $C^\infty$ to a smooth $J$-holomorphic disk~$u_\infty$.
  Similarly, if the nodal limit consists only of one nonconstant
  $J$-holomorphic disk $u_\infty$ and one ghost disk containing the
  interior marked point, then this means that $u_j$ converges in
  $C^\infty$ to $u_\infty$ while $z_j$ converges to a point
  in~$\p\DD^2$.
  In all other cases, $u_j$ can be viewed as converging to a bubble
  tree which may include both spheres and disks, while $z_j$ converges
  to an interior or boundary point on one of the components.
  Suppose now that the sequence $(u_j, z_j)$ converges to a bubble
  tree $(\mathfrak{t}_\infty, z_\infty)$.
  We will show that $\mathfrak{t}_\infty$ does not contain any
  nonconstant disk bubble.
  (Since the boundary marked points are always mapped to distinct
  points in the image, stability then implies that with the exception
  of the cases interpreted above as smooth limits,
  $\mathfrak{t}_\infty$ contains no disk bubbles at all.)
  Suppose on the contrary that the sequence~$u_j$ bubbles a
  nonconstant disk~$v$ with boundary on the \LOB~$L_{\bbf^+_\infty}$.
  The points $1, -1, i$ divide $\p\DD^2$ into three segments, one of
  which is necessarily disjoint from the bubbling region.
  The fact that for each $j$ the function $\vartheta \circ
  \restricted{u_j}{\p\DD^2}$ is a diffeomorphism then implies that
  $\vartheta \circ \restricted{v}{\p\DD^2}$ is not surjective.
  Thus the boundary of $v$ is somewhere tangent a page of the
  \LOB~$L_{\bbf^+_\infty}$, but it follows from a standard argument
  using the boundary point lemma
  \cite[Theorem~II.1.3]{NiederkrugerHabilitation} that the disk~$v$
  cannot exist.
  We conclude that $\mathfrak{t}_\infty$ is a bubble tree containing
  only holomorphic spheres.
  This is case~(3).
  If there is a sequence $w_j \in \DD^2$ such that $u_j(w_j)$
  approaches the binding, then the maps $u_j$ are Bishop disks for $j$
  large enough because, sooner or later, the images of the $u_j$ will
  intersect the domain~$V$ as in Proposition~\ref{prop: bishop disks
    close to binding} nontrivially.
  This implies that the limit~$u_\infty$ is the constant map at a
  point $p_\infty \in B_{\bbf^+_\infty}$, and we have case~(2).
\end{proof}

Using the natural identification of $\mM\bigl(\widehat W, S_B;
J\bigr)$ with the space of parametrized curves satisfying \eqref{eq:
  slice}, we can now compactify $\mM\bigl(\widehat W, S_B; J\bigr)$ by
adding the limiting configurations described in Proposition~\ref{prop:
  Gromov compactness}.
We will denote this compactified moduli space by $\overline
\mM\bigl(\widehat W, S_B; J\bigr)$.
Its ``boundary''
\begin{equation*}
  \p \overline \mM\bigl(\widehat W, S_B; J\bigr) \subset
  \overline \mM\bigl(\widehat W, S_B; J\bigr)
\end{equation*}
can be defined naturally as the set of equivalence classes
$[(\bbf^+,u,z)]$ for which either $\bbf^+ \in \p \DD^2_r$, $z \in \p
\DD^2$ (including cases where the domain of $u$ contains sphere
bubbles), or $u$ is a constant map into the binding of a \LOB.
The compactification can also be decomposed naturally into two
disjoint pieces,
\begin{equation*}
  \overline\mM\bigl(\widehat W, S_B; J\bigr) =
  \mMsmooth \cup \mMbubble,
\end{equation*}
defined as the subsets consisting of non-nodal and nodal curves
respectively.
We define
\begin{equation*}
  \p \mMsmooth := \p \overline\mM\bigl(\widehat W, S_B; J\bigr)
  \cap \mMsmooth \;.
\end{equation*}
The next proposition describes the topology of the compactified moduli
space.
We refer to \cite[\S 6.5]{McDuffSalamonJHolo} for general facts about
pseudocycles, and \cite{Schwarz_equivalences,ZingerPseudoCycles} for
the fact that pseudocycles up to bordism can be identified with
integral homology classes.

\begin{proposition}\label{prop: topology of compactified moduli space}
  Let $(W,\omega)$ be a symplectic filling of a contact
  $(2n-1)$-manifold $(M,\xi)$.
  Suppose that $(M,\xi)$ has been obtained by a surgery of index $k
  \le n-1$ from another contact manifold, and that $\beltSphere{n+m}$
  is the corresponding belt sphere (with $n = k + m + 1$).
  Deform $W$ as described in Section~\ref{sec: deformed subcritical
    handles} to $(\widehat W, \omega)$, and let $\overline
  \mM\bigl(\widehat W, S_B; J\bigr)$ be the compactification of the
  moduli space $\mM\bigl(\widehat W, S_B; J\bigr)$ of disks attached
  to $S_B \subset  \beltSphere{n+m}$.
  Then $\ev \colon \mM\bigl(\widehat W, S_B; J\bigr) \to \widehat W$ 
  extends to a continuous map
  \begin{equation*}
\overline\ev\colon \left(\overline{\mM}\bigl(\widehat W, S_B; J\bigr),
\p \overline{\mM}\bigl(\widehat W, S_B; J\bigr)\right) \to
\left(\widehat W,\beltSphere{n+m}\right).
\end{equation*}
Moreover:
  \begin{itemize}
  \item [(a)] If $(W,\omega)$ is semipositive, then
    \begin{equation*}
      \restricted{\ev}{\p \mMsmooth} \colon \p \mMsmooth \to \widehat{W}
    \end{equation*}
    defines an $(n+m)$-dimensional pseudocycle in $\widehat{W}$
    representing the homology class $\pm [\beltSphere{n+m}] \in
    H_{n+m}(\widehat{W};\ZZ)$, and
    \begin{equation*}
      \restricted{\ev}{\mMsmooth} \colon \mMsmooth \to \widehat{W}
    \end{equation*}
    defines a bordism of the above pseudocycle to the empty
    $(n+m)$-dimensional pseudocycle in~$\widehat{W}$.
  \item [(b)] If $(W, \omega)$ is symplectically aspherical, then
    $\overline \mM\bigl(\widehat W, S_B; J\bigr)$ is homeomorphic to a
    manifold with boundary and corners of the form
    \begin{equation*}
      \overline \Sigma \times \DD^2 \;,
    \end{equation*}
    where $\overline \Sigma$ is a smooth, compact, connected and
    oriented $(n + m - 1)$-manifold with boundary and corners, whose
    boundary is homeomorphic to $\SS^{n+m-2}$.
    Furthermore,
    \begin{equation*}
      \restricted{\ev}{\p\overline \mM\bigl(\widehat W, S_B; J\bigr)} \colon 
      \p\overline \mM\bigl(\widehat W, S_B; J\bigr) \to \beltSphere{n+m}
    \end{equation*}
    is a map of degree~$\pm 1$.
  \end{itemize}
\end{proposition}
\begin{proof}
  Let us describe the natural topology on $\overline \mM\bigl(\widehat
  W, S_B; J\bigr)$, using again the identification of
  $\mM\bigl(\widehat W, S_B; J\bigr)$ with the slice in $\widetilde
  \mM\bigl(\widehat W, S_B; J\bigr)$ defined via the
  conditions~\eqref{eq: slice}.
  The boundary of the uncompactified space $\mM\bigl(\widehat W, S_B;
  J\bigr)$ consists of holomorphic disks whose marked points lie in
  $\p\DD^2$, together with Bishop disks with boundary on a
  \LOB~$L_{\bbf^+}$ with $\bbf^+ \in \p \DD^2_r$.
  Proposition~\ref{prop: Gromov compactness}, provides a description
  of the two additional limit objects we need to consider.
  If $(\bbf^+, p_\infty, z)$ is one of the constant disks appearing in
  case~(2) of Proposition~\ref{prop: Gromov compactness}, then it
  follows from Proposition~\ref{prop: bishop disks close to binding}
  that it is surrounded only by Bishop disks.
  Using the parametrization of the Bishop disks given at the beginning
  of Section~\ref{sec: bishop disks at the boundary}---in this
  description the constant disks are obtained by choosing $C = 0$---we
  can add the constant disks $(\bbf^+, p_\infty, z)$ to the chosen
  slice in $\widetilde \mM\bigl(\widehat W, S_B; J\bigr)$, and give it
  a smooth structure that agrees with the one induced by
  $C^\infty$-convergence of maps.
  Attaching the constant disks in this way corresponds to adding
  boundary points to the global slice.
  Defining a smooth structure on $\overline \mM\bigl(\widehat W, S_B;
  J\bigr)$ in this way, it is straightforward to see that the
  evaluation map extends smoothly to the constant disks in $\overline
  \mM\bigl(\widehat W, S_B; J\bigr)$.
  The other singular points we need to consider in $\overline
  \mM\bigl(\widehat W, S_B; J\bigr)$ are bubble trees, each consisting
  of one holomorphic disk and several holomorphic spheres.
  If $(W, \omega)$ is symplectically aspherical as in case~(b), it
  does not contain any holomorphic spheres, and hence no bubbles can
  appear.
  In this case, the compactified moduli space $\overline
  \mM\bigl(\widehat W, S_B; J\bigr)$ will be diffeomorphic to
  \begin{equation*}
    \overline {\Sigma} \times \DD^2 \;
  \end{equation*}
  according to Propositon~\ref{prop: topology of the uncompactified
    moduli space}, where $\overline{\Sigma}$ is a smooth compact
  manifold with boundary and corners.
  If we are in case~(a), then bubbling of spheres may occur, but
  standard index counting arguments using the semipositivity
  assumption imply that such bubbling is a ``codimension~$2$
  phenomenon''.
  The restriction of $\ev$ to $\p \mMsmooth$ is then a pseudocycle,
  and the restriction to $\mMsmooth$ is a bordism of this pseudocycle
  to the trivial one.
  In the absence of bubbling, $\overline{\mM}\bigl(\widehat W, S_B;
  J\bigr)$ is a trivial disk bundle over the compact base
  manifold~$\overline{\Sigma}$, whose boundary consists of Bishop
  disks sitting on boundary \LOB{}s and/or collapsing into the
  binding.
  The boundary \LOB{}s are parametrized by $\bbf^+ \in \p\DD^m_r \cong
  \SS^{m-1}$, and there is precisely one (unparametrized) Bishop disk
  going through every point of the page of the Legendrian open book of
  $L_{\bbf^+}$, hence we conclude that the first disks can be
  parametrized by $\SS^{m-1} \times \DD^{n-1}$.
  The binding of a \LOB~$L_{\bbf^+}$, on the other hand, is
  diffeomorphic to $\SS^{n-2}$, and since there is a $\DD^m_r$-worth
  of \LOB{}s, we conclude that $\overline{\mM}\bigl(\widehat W, S_B;
  J\bigr)$ contains a family of constant disks that is parametrized by
  $\DD^m_r \times \SS^{n-2}$.
  These two parts meet at their boundaries and form a (topological)
  manifold homeomorphic to $\SS^{n + m - 2}$, as claimed.
  Finally, we observe that the restriction of $\overline{\ev}$ to $\p
  \overline{\mM}\bigl(\widehat W, S_B; J\bigr)$ is always bijective on
  some subset consisting of Bishop disks, so it is a map of degree
  $\pm 1$ onto $\beltSphere{n+m}$ whenever $\p
  \overline{\mM}\bigl(\widehat W, S_B; J\bigr)$ is a topological
  manifold.
  More generally, this implies that the pseudocycle
  $\restricted{\ev}{\p \mMsmooth} \colon \p \mMsmooth \to
  \beltSphere{n+m}$ represents a generator of
  $H_{n+m}(\beltSphere{n+m};\ZZ)$ whenever it is well defined.
  The orientability of $\mM\bigl(\widehat W, S_B; J\bigr)$ is shown in
  Appendix~\ref{sec: orientability of the moduli space}.
\end{proof}

\begin{proof}[Proof of Theorem~\ref{thm: main theorem}]
  The proof of statement~(a) follows directly from part~(a) of
  Proposition~\ref{prop: topology of compactified moduli space}, using
  the natural identification between singular homology classes and
  bordism classes of pseudocycles, see
  \cite{Schwarz_equivalences,ZingerPseudoCycles}.
  Statement~(b) can be obtained by using part~(b) of
  Proposition~\ref{prop: topology of compactified moduli space}.
  Since $\overline \mM\bigl(\widehat W, S_B; J\bigr)$ is diffeomorphic
  to a trivial disk bundle, we can apply Proposition~\ref{prop:
    surgering high manifold times disk} in the general situation, or
  Propositions~\ref{prop: surgering surface times disk} and~\ref{prop:
    surgering 3-manifold times disk} respectively, when $n+m = 3$ or
  $n+m = 4$.
  This implies that $\beltSphere{n+m}$ is the trivial element in the
  oriented bordism group $\Omega^{SO}_{n+m}(W)$, and it is even
  contractible in $W$, if $n+m = 3$ or $4$ as claimed.
\end{proof}

\section{Surgery on moduli spaces}\label{sec: surgery on the moduli
  space}

If bubbling can be ruled out in the proof of Theorem~\ref{thm: main
  theorem} given in the previous section, we can use topological
results to conclude that the belt sphere is not only nullhomologous
but null-bordant in $\Omega_*^{SO}(\widehat W)$, and in cases where
$\dim \mM\bigl(\widehat W, S_B; J\bigr) \le 5$, it is even
null-homotopic.
The idea in both cases is to attach handles to the moduli space and
extend the evaluation map accordingly so that we obtain a new space
together with a map into $\widehat W$ which will be topologically
simpler than the original moduli space.
Our argument for this will make essential use of the fact that the
moduli space is naturally a trivial disk bundle.
Note that if $\Sigma$ is a compact oriented $k$-manifold with
boundary, then using handle attachments to turn $\Sigma \times \DD^2$
into a ball cannot succeed unless $\Sigma \times \DD^2$ admits an
embedding into $\RR^{k+2}$, which cannot always be guaranteed, i.e.~in
general there are topological obstructions to the applicability of
this technique to obtain contractibility of the belt sphere.
We will show that these can be overcome if $\dim \Sigma \le 3$.

\begin{proposition}\label{prop: surgering high manifold times disk}
  Let $W$ be a compact manifold possibly with boundary, and let $S
  \subset W$ be an embedded $(k+1)$-sphere.
  Assume that $\Sigma$ is a compact connected orientable $k$-manifold
  with non-empty boundary.
  Let $X$ be $\Sigma \times \DD^2$, and let
  \begin{equation*}
    f\colon (X, \p X) \to (W,S)
  \end{equation*}
  be a continuous map, whose restriction to the boundary
  \begin{equation*}
    \restricted{f}{\p X}\colon \p X \to S
  \end{equation*}
  is of degree~$1$.
  Then it follows that $S$ is null-bordant in $\Omega^{SO}_{k+1}(W)$.
\end{proposition}
\begin{proof}
  We will assume $k\ge 2$, since for $k\le 1$, $\Sigma$ is either a
  point or a closed interval.
  The manifold~$\Sigma$ has a handle decomposition that consists of a
  $k$-disk with finitely many handles of index~$<k$ attached to it
  (attached successively in order of their indices).
  Since the product of a $k$-dimensional handle with $\DD^2$ is a
  $(k+2)$-dimensional handle of the same index, $X = \Sigma \times
  \DD^2$ is built up by attaching $(k+2)$-dimensional handles of index
  $<k$.
  For each $q=0,\dotsc,k-1$, let $X^{(q)} \subset X$ denote the union
  of all the handles up to index~$q$, so $X = X^{(k-1)}$.
  For every $1$-handle in $X^{(1)}$ we can find a closed curve in the
  boundary of $X^{(1)}$ by pushing the core of the handle into $\p
  X^{(1)}$, and connecting the end points with a path in $\p X^{(1)}$
  that does not intersect any other $1$-handle.
  This is possible, because $X^{(0)}$ consists of a unique $0$-handle
  that in particular is connected.
  These curves intersect the belt sphere of the corresponding handle
  exactly once.
  Moreover since $\dim \p X^{(1)} = k+1$, there is enough space to
  assume that the loops corresponding to different $1$-handles are
  disjoint from each other and also disjoint from any of the attaching
  circles of the $2$-handles needed to obtain $X^{(2)}$.
  The loops thus embed into $\p X^{(2)}$, and we can repeat this
  reasoning to see that they also embed into $\p X^{(3)}$ and so on up
  to $\p X$.
  Standard Morse theory implies that a $q$-handle can be canceled out
  by attaching a $(q+1)$-handle along an embedded $q$-sphere that
  intersects the belt sphere of the $q$-handle exactly once.
  It is thus possible to convert $X^{(1)}$ into a ball by attaching
  $2$-handles, each corresponding to one of the $1$-handles.
  Since the loops also embed into $\p X$, we may equally well attach
  the $2$-handles to $X$, obtaining in this way a compact connected
  orientable $(k+2)$-manifold~$X'$ that admits a handle decomposition
  with exactly one $0$-handle, and no handles of index~$1$ or of
  index~$\ge k$.
  The map $f \colon X \to W$ can be extended to the newly added
  $2$-handles because $\p X$, and hence also the attaching curves, are
  mapped to $S$, which is simply connected.
  We can thus construct a continuous map
  \begin{equation*}
    f'\colon \bigl(X', \p X'\bigr) \to (W,S)  \;.
  \end{equation*}
  The restriction of $f'$ to the boundary is still of degree~$1$,
  because the image of the $2$-handles can be assumed to be a thin set
  in the $(k+1)$-sphere~$S$.
  To finish the proof, consider the double of $X'$, obtained by gluing
  a copy of $X'$ with reversed orientation to itself along its
  boundary.
  One can decompose the double into handles such that for each
  $q$-handle in the original $X'$, there is a corresponding
  $(k+2-q)$-handle in the second copy (think of the two copies of $X'$
  as carrying Morse functions $f$ and $c - f$ for some constant $c \in
  \RR$).
  Our handle decomposition of the double therefore has exactly one
  handle of index $k+2$ and none of index $k+1$.
  Let $\widehat{X}$ denote the result of removing the $(k+2)$-handle.
  Then $\widehat{X}$ is a compact connected orientable
  $(k+2)$-manifold obtained from $X'$ by attaching additional handles
  of various indices $3,\dotsc, k$, and $\p\widehat{X} \cong
  \SS^{k+1}$.
  Since $\pi_q(S) = 0$ for all $q=2,\dotsc,k$, the map~$f'$ extends
  from $X'$ to a continuous map
  \begin{equation*}
    \widehat f\colon
    \bigl(\widehat X, \p \widehat X\bigr) \to (W,S)  \;,
  \end{equation*}
  mapping all of the additional handles into $S$.
  Note in particular that the restriction of $\widehat f$ to $\p
  \widehat X$ is a degree~$1$ map to $S$, because all the handles
  added to $X'$ are of index lower than $\dim S$.
  The restriction $\restricted{\widehat f}{\p \widehat X}$ is
  homotopic to a diffeomorphism between the two $(k+1)$-spheres~$\p
  \widehat X$ and $S$, and it follows that $S$ is null-bordant in
  $\Omega^{SO}_{k+1}(W)$.
\end{proof}

\begin{proposition}\label{prop: surgering surface times disk}
  Let $W$ be a compact manifold possibly with boundary, and let $S
  \subset W$ be an embedded $3$-sphere.
  Assume that $\Sigma$ is a compact connected orientable surface with
  non-empty boundary.
  Let $X$ be $\Sigma \times \DD^2$, and let
  \begin{equation*}
    f\colon (X, \p X) \to (W,S)
  \end{equation*}
  be a continuous map, whose restriction to the boundary
  \begin{equation*}
    \restricted{f}{\p X}\colon \p X \to S
  \end{equation*}
  is of degree~$1$.
  Then it follows that $S$ is contractible in $W$.
\end{proposition}
\begin{proof}
  The proof is a special case of the argument used for
  Proposition~\ref{prop: surgering high manifold times disk}.
  The difference is that $X$ consists only of a $4$-disk and
  $1$-handles, thus after passing to $X'$ by attaching $2$-handles, it
  already follows that $X'$ is diffeomorphic to a $4$-disk~$\DD^4$.
  We then obtain a continuous map
  \begin{equation*}
    f'\colon \bigl(\DD^4, \p\DD^4\bigr) \to (W,S)
  \end{equation*}
  whose restriction to the boundary is still of degree~$1$, and is
  therefore homotopic to a homeomorphism $\p\DD^4 \to S$.
\end{proof}

\begin{proposition}\label{prop: surgering 3-manifold times disk}
  Let $W$ be a compact manifold, possibly with boundary, and let $S
  \subset W$ be an embedded $4$-sphere.
  Let $\Sigma$ be a compact connected orientable $3$-manifold with
  non-empty boundary $\p\Sigma \cong \SS^2$.
  Assume that $X$ is $\Sigma \times \DD^2$, and that
  \begin{equation*}
    f\colon (X, \p X) \to (W,S)
  \end{equation*}
  is a continuous map, whose restriction to the boundary
  \begin{equation*}
    \restricted{f}{\p X}\colon \p X \to S
  \end{equation*}
  is of degree~$1$.
  Then it follows that $S$ is contractible in $W$.
\end{proposition}
\begin{proof}
  The manifold~$\Sigma$ is an orientable $3$-manifold minus a ball.
  It admits a handle decomposition given by a $3$-ball with $1$- and
  $2$-handles attached, and we may assume that $\Sigma$ has been
  obtained by attaching first the $1$-handles and then the
  $2$-handles.
  As in the proof of Proposition~\ref{prop: surgering high manifold
    times disk}, it follows that the manifold~$X$ is built up by first
  attaching $1$- and then $2$-handles to a $5$-ball~$\DD^5$.
  Denote by $X^{(1)}$ the intermediate space consisting only of the
  $5$-ball and the $1$-handles.
  It is easy to find for every $1$-handle an embedded loop in $\p
  X^{(1)}$ that intersects the belt sphere of the handle exactly once.
  For dimensional reasons, these loops will be generically disjoint
  from each other, but they will also generically not intersect any of
  the attaching circles of the $2$-handles.
  We can cancel all $1$-handles of $X^{(1)}$ by attaching $2$-handles
  along the chosen loops.
  The chosen loops also embed into $\p X$, hence we can also kill all
  $1$-handles by attaching $2$-handles to $X$.
  Note that we could get rid of the $1$-handles without choosing a
  particular framing when attaching the $2$-handles; however $X$ is
  parallelizable (as is any oriented $3$-manifold), and using
  Lemma~\ref{lemma: attaching 2-handles with trivial tangent bundle}
  below, we attach the $2$-handles in such a way that the resulting
  manifold is also parallelizable.
  Since the image~$f(\p X)$ lies in the sphere~$S$, we can extend $f$
  to the additional $2$-handles without changing the degree of
  $\restricted{f}{\p X}$.
  After the previous step, we will assume that $X$ is a $5$-manifold
  with trivial tangent bundle that has been obtained by gluing
  $2$-handles to the $5$-ball.
  Every embedding of $\SS^1$ into $\SS^4 = \p\DD^5$ is isotopic to a
  standard one for dimensional reasons (see
  \cite{HaefligerHighDimensionalKnots}), and it follows that the
  $2$-handles are all attached along unknots.
  Note also that these unknots are unlinked since we may shrink the
  first unknot into an arbitrarily small ball, so that the other loops
  will bound embedded disks that are disjoint from this ball.
  We may therefore assume that $X$ is the boundary sum of a finite
  collection of $5$-manifolds, each consisting of a $5$-ball with a
  single $2$-handle attached along an unknot.
  The only invariant of each such manifold is the framing of the
  $2$-handle.
  It is given by a loop in $\SO(3)$, which means there are only two
  choices, corresponding to the two elements of $\pi_1(\SO(3))$.
  In fact, each of these manifolds is diffeomorphic to either the
  trivial rank~$3$ bundle over $\SS^2$ or the twisted one, $\SS^2
  \tilde\times \DD^3$.
  The total space of the twisted one is not parallelizable:
  it suffices to study $\restricted{T \bigl(\SS^2 \tilde\times
    \DD^3\bigr)}{\SS^2 \tilde\times \{\0\}}$ which is obtained by
  clutching two copies of $\CC\oplus \RR^3$ over two disks
  together.
  The gluing map is $e^{2i\phi} \oplus \psi$, where $\psi$ is the
  nontrivial loop in $\pi_1(\SO(3))$, but since this is the nontrivial
  element of $\pi_1(\SO(5))$, the bundle is not trivial.
  It follows that $X$ is the boundary connected sum of copies of
  $\SS^2 \times \DD^3$.
  We can then also kill the $2$-handles by attaching $3$-handles, and
  the map~$f$ extends to this new manifold.
  This proves that $S$ is homotopically trivial in the filling.
\end{proof}

The following lemma was used above in the proof of
Proposition~\ref{prop: surgering 3-manifold times disk}.

\begin{lemma}\label{lemma: attaching 2-handles with trivial tangent
    bundle}
  Let $X$ be a compact parallelizable $n$-manifold with boundary, and
  let $\gamma$ be an embedded loop in $\p X$.
  Assume $n\ge 5$.
  Then one can choose a framing of $\gamma$ such that the manifold
  \begin{equation*}
    X \cup_\gamma H_2
  \end{equation*}
  obtained by attaching a $2$-handle~$H_2$ along $\gamma$ is also
  parallelizable.
\end{lemma}
\begin{proof}
  A framing of $\gamma$ is an oriented trivialization of the normal
  bundle~$\nu(\gamma)$ of $\gamma$ in $\p X$.
  Given one framing, any other one can be obtained by multiplying the
  first one in each fiber with a matrix in $\GL^+(n-2)$, i.e. the
  second framing can be represented with respect to the first one by a
  map $\SS^1 \to \GL^+(n-2)$.
  We are only interested in framings up to homotopy, hence it follows
  that all framings are classified by $\pi_1\bigl(\GL^+(n-2)\bigr)$,
  and since $\GL^+(n-2) \simeq \SO(n-2)$, there are only two possible
  choices.
  Choose now a trivialization of $TX$.
  Such a trivialization allows us to identify
  $\restricted{TX}{\gamma}$ with $\SS^1 \times \RR^n$.
  Any other trivialization of $\restricted{TX}{\gamma}$ can be
  represented with respect to the first one by a map $\SS^1 \to
  \GL^+(n)$, that is, up to homotopy there are also only two
  trivializations of $\restricted{TX}{\gamma}$, corresponding to the
  elements of $\pi_1(\SO(n))$.
  In particular, any framing $(e_1, \dotsc, e_{n-2})$ of $\nu(\gamma)$
  extends to a basis $(f_1, f_2, e_1, \dotsc , e_{n-2})$ of
  $\restricted{TX}{\gamma}$, where the vector fields $f_1$ and $f_2$
  are given by
  \begin{align*}
    f_1 &= \vec n\, \cos\phi - \dot \gamma\, \sin\phi  \\
    f_2 &= \vec n\, \sin\phi + \dot \gamma\, \cos\phi \; ,
  \end{align*}
  where $\phi$ parametrizes $\gamma$.
  Here $\vec n$ denotes the outward normal vector field to the
  boundary~$\p X$, and $\dot \gamma$ is the tangent vector field to
  the loop~$\gamma$.
  If this basis is not homotopic to the given trivialization of
  $\restricted{TX}{\gamma}$, it suffices to choose instead
  \begin{equation*}
    \bigl(f_1, f_2, e_1 \cos\phi - e_2 \sin\phi,
    e_1 \sin\phi + e_2 \cos\phi,  e_3, \dotsc, e_{n-2}\bigr) \;,
  \end{equation*}
  which corresponds to the second framing of $\gamma$, but also to the
  other homotopy class of possible trivializations of
  $\restricted{TX}{\gamma}$.
  It is thus possible to homotope the trivialization of $TX$ into one
  that coincides close to $\gamma$ with $(f_1, f_2, e_1, \dotsc,
  e_{n-2})$, where $(e_1,\dotsc,e_{n-2})$ is a framing of $\gamma$.
  On $H_2 = \DD^2 \times \DD^{n-2}$ with coordinates $(x,y;\z) \in
  \DD^2 \times \DD^{n-2}$, the attaching circle $\{x^2 + y^2 = 1,\, \z
  = \0\}$ has the obvious framing
  $\bigl(\partial_u, \partial_v, \partial_{z_1},
  \dotsc, \partial_{z_{n-2}}\bigr)$.
  If we glue $H_2$ to $X$ with the chosen framing, then the
  trivialization $(f_1, f_2, e_1, \dotsc, e_{n-2})$ extends to
  $\bigl(\partial_u, \partial_v, \partial_{z_1},
  \dotsc, \partial_{z_{n-2}}\bigr)$, so the manifold $X\cup_\gamma
  H_2$ has trivial tangent bundle, as desired.
\end{proof}

\section{Contact structures that are not contact connected 
sums}
\label{sec: almost contact surgery not genuine surgery}

In this section we prove Theorem~\ref{thm: almost contact connSum not
  genuine}.
The construction we are going to use is inspired by a similar one in
\cite{BowdenCrowleyStipsicz2}, though we do not need the full strength
of that paper.
Let $M$ be a closed $(2n-1)$-dimensional manifold that admits an
almost contact structure and that has a handle decomposition with a
single handle of index~$0$, a single one of index~$2n-1$, and
otherwise only handles of indices~$n-1$ and $n$.
We assume also that $M$ is not a homotopy sphere, which by the
Hurewicz theorem implies that it must have nontrivial homology in
dimension $n-1$ or~$n$.
Possible examples include the unit cotangent bundle of $\SS^n$, and
$\SS^{n-1} \times \SS^n$; the first carries a canonical contact
structure, and the second is easily seen to be almost contact since it
is stably parallelizable.
Remove a small open disk $D$ from $M$ and denote the resulting
manifold by $M^*$.
The product manifold~$W = M^* \times [-1,1]$ is compact and has
boundary and corners, and after smoothing, its boundary
\begin{equation*}
  \p W = M^*\times \{-1\}
  \cup \bigl(\p M^*\times [-1,1]\bigr)  \cup M^*\times \{+1\}
\end{equation*}
is diffeomorphic to $M \conSum (-M)$.
Now we can proceed with the proof of Theorem~\ref{thm: almost contact
  connSum not genuine}.

\begin{proof}[Proof of Theorem~\ref{thm: almost contact connSum not
    genuine}]
  By assumption, the manifold~$M^*$ admits a Morse function with
  outward pointing gradient at the boundary and critical points of
  index at most~$n$, and the same is therefore true of~$W$.
  Moreover, any almost contact structure $\Xi$ on $M$ induces an
  almost complex structure on $W$, thus by a well-known theorem of
  Eliashberg \cite{Eliashberg_Stein}, there is a Stein structure whose
  complex structure is homotopic to the given almost complex
  structure.
  The boundary $\p W \cong M \conSum (-M)$ inherits from this Stein
  structure a contact structure~$\xi$ which is homotopic to the almost
  contact structure $\Xi \conSum \overline{\Xi}$.
  Note that the belt sphere of the connected sum (i.e.~the center of
  the ``neck'' in $M \conSum (-M)$) is
  \begin{equation*}
    S := \p M^* \times \{0\} \subset \p W \;.
  \end{equation*}
  Arguing by contradiction, suppose now that $\xi_1$ and $\xi_2$ are
  positive contact structures on $M$ and $-M$ respectively such that
  $\xi_1 \conSum \xi_2$ is isotopic to~$\xi$.
  Then after a deformation of the Stein structure near $\p W$ and
  hence an isotopy of~$\xi$, we can assume $\xi$ in a neighborhood of
  $S$ is contactomorphic to the contact structure on a neighborhood of
  the belt sphere of an index one Weinstein handle.
  According to Proposition~\ref{prop: topology of compactified moduli
    space} there is a compact $(2n-1)$-dimensional moduli space
  $\overline{\mM}\cong \overline{\Sigma} \times \DD^2$ and an
  evaluation map $\ev \colon (\overline{\mM}, \p \overline{\mM}) \to
  (W, S)$ such that
  \begin{enumerate}
  \item $\restricted{\ev}{\p \overline{\mathcal M}} \colon \p
    \overline{\mM} \to S$ has degree one, and
  \item $\overline \Sigma$ is a compact orientable $(2n-3)$-manifold
    with non-empty boundary.
  \end{enumerate}
  Moreover, $\ev$ is a diffeomorphism on some open subset.
  Consider the projection $p \colon W = M^* \times [-1,1] \to M^*$,
  which maps $S$ to $\p M^*$, and denote by $f \colon \overline{\mM}
  \to M^*$ the composition $f = p \circ \ev$.
  It is easy to check that $f \colon \overline{\mM} \to M^*$ now
  satisfies the conditions of Lemma~\ref{lemma: map of product
    manifold induces contractibility}.
  This implies that $M^*$ has vanishing homology in positive degrees,
  and is thus a contradiction.
\end{proof}

On the other hand, note that there is no homotopical obstruction to
decomposing $\bigl(M \conSum (-M), \xi\bigr)$, because $\xi$ is, by
construction, homotopic to $\Xi \conSum \overline{\Xi}$.

\begin{remark}
  A similar argument can be used to find examples of Stein fillable
  contact structures that are homotopic (through almost contact
  structures) but not isotopic to contact structures obtained via
  subcritical surgery of arbitrary index $k=1,\dotsc,n-1$.
  The above is the $k=1$ case of this result.
\end{remark}

\section{The Weinstein conjecture for subcritical surgeries}
\label{sec:Weinstein}

We will now prove Theorem~\ref{thm:Weinstein}, the existence of
contractible Reeb orbits for certain contact manifolds $(M',\xi')$
obtained by subcritical surgery.
Under either of the first two conditions stated in the theorem, the
proof is a trivial modification of the proof of Theorem~\ref{thm: main
  theorem}, following \cite{HoferWeinstein}.
Suppose $\alpha$ is the contact form for which we'd like to find a
contractible Reeb orbit, and let $\alpha'$ denote a second contact
form that matches the one given in our Weinstein surgery model near
the belt sphere $\beltSphere{2n-k-1}$.
After rescaling $\alpha$, we can find an exact symplectic structure on
$\RR \times M'$ that matches $d(e^t \alpha)$ on $(-\infty,-1] \times
M'$ and $d(e^t \alpha')$ on $[-1/2,\infty) \times M'$.
We then choose a compatible almost complex structure and, as in
Theorem~\ref{thm: main theorem}, study the moduli space of holomorphic
disks in $\RR \times M'$ with boundary in the \LOB{}s obtained by
deformation from $\{0\} \times \beltSphere{2n-k-1} \subset \RR \times
M'$.
If $\alpha$ admits no contractible Reeb orbits, then bubbling is
impossible, so the proof of Theorem~\ref{thm: main theorem} shows that
$\beltSphere{2n-k-1}$ will be null-bordant in $\RR \times M'$, and
thus also in~$M'$.
If $n=3$, or $n=4$ with $k=3$, it also shows that $\beltSphere{2n-k-1}$
is trivial in $\pi_{2n-k-1}(\RR \times M') = \pi_{2n-k-1}(M')$.
It remains to handle the third condition in
Theorem~\ref{thm:Weinstein}, which specifically concerns contact
connected sums in dimension five.
The above argument shows that in this situation, if there is no
contractible Reeb orbit, then the belt sphere must be nullhomotopic.
But the following theorem of Ruberman \cite{Ruberman} says that this
can only happen in the cases excluded by the third condition.

\begin{theorem}[Ruberman]\label{thm: con_sum contractible belt sphere}
  Let $M$ be a closed oriented manifold, and suppose $S$ is an
  embedded codimension~$1$ sphere that is nullhomotopic.
  Then either $S$ is the boundary of a homotopy-ball embedded in $M$,
  or $M$ is the connected sum
  \begin{equation*}
    M = N_0 \conSum N_1
  \end{equation*}
  of two rational homology spheres~$N_0$ and $N_1$, one of which is
  simply connected, while the other has finite fundamental group.
\end{theorem}

The proof of Theorem~\ref{thm:Weinstein} is thus complete.

\appendix

\section{Orientability of the moduli spaces}
\label{sec: orientability of the moduli space}

In this appendix we prove that the moduli spaces used in this paper
are orientable.
Let us fix some notation which will be used in the proof.
Fix a real number $p > 2$.
Let $\bB$ denote the space of pairs $(\bbf^+, u)$ where:
\begin{itemize}
\item $\bbf^+ \in \DD^m_r$ with $\norm{\bbf^+} <r$;
\item $u \colon \DD^2 \to W$ is a map of class $W^{1,p}$ such that
  $u(\p \DD^2) \subset L_{\bbf^+} \setminus B_{\bbf^+}$ where
  $L_{\bbf^+}$ denotes the \LOB indexed by $\bbf^+$ and $B_{\bbf^+}$
  its binding; and
\item $\vartheta \circ \restricted{u}{\p \DD^2}$ has degree one, where
  $\vartheta\colon L_{\bbf^+} \setminus B_{\bbf^+} \to \SS^1$ is the
  fibration of the \LOB.
\end{itemize}
Of course the information about $\bbf^+$ is already contained in $u$,
and $\bbf^+$ only serves for book-keeping.
We denote by $\bB_{\bbf^+}$ the fibers of the projection $\mathfrak{p}
\colon \bB \to \DD^m_r$, i.e. $\mathfrak{p}(\bbf^+, u) = \bbf^+$.
Then $\bB_{\bbf^+}$ consists of the maps $u \in \bB$ such that $u(\p
\DD^2) \subset L_{\bbf^+} \setminus B_{\bbf^+}$.
The linearized Cauchy-Riemann operator at $(\bbf^+, u) \in \bB$ (or to
be more precise, the \emph{vertical differential} of the nonlinear
Cauchy-Riemann operator, as defined in \cite{McDuffSalamonJHolo}),
will be denoted $\widetilde D_{(\bbf^+, u)}$.
Recall that this depends on a choice of connection on~$W$, though it
is independent of this choice whenever $u$ is $J$-holomorphic.
We define 
\begin{equation*}
  \det \widetilde D_{(\bbf^+, u)} = \Lambda^{top} \ker \widetilde D_{(\bbf^+, u)}
  \otimes \Lambda^{top} (\mathrm{coker} \, \widetilde D_{(\bbf^+, u)})^* \;.
\end{equation*}
The \defin{determinant bundle} $\dD \to \bB$ is the real rank-one 
bundle whose fiber at $(\bbf^+, u)$ is $\det \widetilde D_{(\bbf^+, u)}$.
In order to prove that the moduli space $\widetilde \mMint(\widehat W,
S_B ; J)$ is orientable, it suffices to show that $\dD \to \bB$ is
trivial.
To better understand the determinant bundle we take a closer look at
the linearized Cauchy-Riemann operator.
The tangent space $T_{(\bbf^+, u)}\bB$ consists of sections $\xi \in
W^{1,p}(u^*T \widehat W)$ such that, for all $z \in \p \DD^2$, they
satisfy $\xi(z) \in T_{u(z)}\beltSphere{n+m}$ and moreover the
projection of $\xi(z)$ to $T_{\bbf^+} \DD^m_r$ is independent of $z
\in \p\DD^2$.
It contains the subspace $T_u \bB_{\bbf^+} \subset T_{(\bbf^+, u)}\bB$
which is defined as
\begin{equation*}
  T_u \bB_{\bbf^+} = \bigl\{ \xi \in W^{1,p}(u^*T\widehat W) \bigm|\,
  \xi(z) \in T_{u(z)} L_{\bbf^+} \, \forall z \in \p \DD^2 \bigr\}
\end{equation*}
and therefore we can identify
\begin{equation}\label{eq: decomposition of TB}
  T_{(\bbf^+, u)}\bB \cong T_u \bB_{\bbf^+} \oplus T_{\bbf^+}\DD^m_r \;.
\end{equation}
The tangent spaces $T_u \bB_{\bbf^+}$ are the fibers of a vector
bundles over $\bB$ which we will denote $T^{\mathrm{vert}}\bB$ and the
decomposition \eqref{eq: decomposition of TB} globalizes to a bundle
isomorphism
\begin{equation*}
  T\bB \cong T^{\mathrm{vert}}\bB \oplus \mathfrak{p}^* T \DD^m_r \;.
\end{equation*}
Although the above isomorphism is not canonical, its homotopy class
is.
We denote the restriction of $\widetilde D_{(\bbf^+, u)}$ to $T_u
\bB_{\bbf^+}$ by $D_{(\bbf^+, u)}$.
If we write the elements of $T_{(\bbf^+, u)}\bB$ as pairs $(\xi, v)
\in T_u \bB_{\bbf^+} \oplus T_{\bbf^+}\DD^m_r$ using the
identification in Equation~\eqref{eq: decomposition of TB}, we can
decompose $\widetilde D_{(\bbf^+, u)}$ as
\begin{equation*}
  \widetilde D_{(\bbf^+, u)}(\xi, v) = D_{(\bbf^+, u)}(\xi) + K_{(\bbf^+, u)}(v) \;.
\end{equation*}
The operator~$D_{(\bbf^+, u)}$ is the linearization at $u$ of the
nonlinear Cauchy-Riemann operator defined on $\bB_{\bbf^+}$, and
therefore it is a linear Cauchy-Riemann type operator.
Let $\dD' \to \bB$ the real line bundle whose fiber at $(\bbf^+, u)$
is $\det D_{(\bbf^+, u)}$.
Since the determinant line bundles of homotopic families of Fredholm
operators are isomorphic, we obtain an isomorphism
\begin{equation}\label{eq: printroom}
  \dD \cong \dD' \otimes \mathfrak{p}^* \Lambda^m T\DD^m_r
\end{equation}
by homotoping the operators $K_{(\bbf^+, u)}$ to the zero operator via
a linear homotopy.
Note that this defines a homotopy of families of Fredholm operators
because the operators~$K_{(\bbf^+, u)}$ are defined on a finite
dimensional space.
By the isomorphism \eqref{eq: printroom}, the triviality of $\dD$ is
equivalent to the triviality of $\dD'$, so from now on we will
concentrate on this second bundle.
Triviality of a rank-one real line bundle can be checked on loops.
Thus let $(\bbf^+_\bullet, u_\bullet) \colon \SS^1 \to \bB$ be a loop
in $\bB$, i.e. $\theta \mapsto (\bbf^+_\theta, u_\theta)$.
From a different point of view we have a map $\tilde{u} \colon \SS^1
\times \DD^2 \to W$ defined as $\tilde{u}(\theta, z) = u_\theta(z)$.
We define the vector bundle $T^{\mathrm{vert}} \beltSphere{n+m}$ such
that $T^{\mathrm{vert}}_p\beltSphere{n+m} = T_pL_{\bbf^+(p)}$, where
$L_{\bbf^+(p)}$ denotes the \LOB containing $p$.
We define a complex vector bundle $E \to \SS^1 \times \DD^2$ by $E=
\tilde{u}^*TW \oplus \underline{\CC}$ and a real vector subbundle $F
\to \SS^1 \times \p \DD^2$ of $\restricted{E}{\SS^1 \times \p \DD^2}$
by
% $F_{(\theta, z)}= T_{u_\theta(z)}L_{\bbf^+_\theta} \oplus \RR$
% for all $(\theta, z) \in \SS^1 \times \partial \DD^2$.
$F=\bigl(\restricted{\tilde{u}}{\SS^1 \times \p
  \DD^2}\bigr)^*T^{\mathrm{vert}}\beltSphere{n+m} \oplus
\underline{\RR}$.
Here $\underline{\CC}$ and $\underline{\RR}$ denote the trivial
complex and real line bundle, respectively.
We denote by $E_\theta$ the restriction of $E$ to $\{\theta\} \times
\DD^2$, by $F_\theta$ the restriction of $F$ to $\{\theta\} \times
\DD^2$ and by $\Gamma(E_\theta, F_\theta)$ the sections of $E_\theta$
which take values in $F_\theta$ along $\p \DD^2$.
Similarly, let $\Gamma(\underline{\CC}, \underline{\RR})$ denote the
sections of the trivial line bundle~$\underline{\CC}$ over $\DD^2$
with real values at $\p \DD^2$, and $D_0$ the standard Cauchy-Riemann
operator acting on $\Gamma(\underline{\CC}, \underline{\RR})$.
We consider the family of linear Cauchy-Riemann type operators
$D^+_\theta = D_{u_\theta}' \oplus D_0$ acting on $\Gamma(E_\theta,
F_\theta)$.
This family gives rise to a determinant line bundle $\dD^+
\to \SS^1$.
Since $\det D_0 = \RR$, we have $\dD^+ \cong u^*_\bullet\dD'$.
Therefore studying the orientability of $\dD^+$ is equivalent to
studying the orientability of the moduli space.
Being spheres, the \LOB{}s are stably parallelizable and the parameter
space $\DD^m_r$ is contractible, so $T^{\mathrm{vert}}
\beltSphere{n+m}$ is stably trivial.
An orthonormal trivialization of $T^{\mathrm{vert}}\beltSphere{n+m}
\oplus \underline{\RR}$ can be pulled back to an orthonormal
trivialization~$\nu_0$ of $F$.
We can also regard $\nu_0$ as a unitary trivialization of
$\restricted{E}{\SS^1 \times \p \DD^2}$, because $F$ is a Lagrangian
subbundle of $\restricted{E}{\SS^1 \times \p \DD^2}$.
However $\nu_0$ does not extend to a unitary trivialization of $E$.
In fact it does not extend to the meridian disks of $\SS^1 \times
\DD^2$, because $(E_\theta, F_\theta)$ has Maslov index two.
We choose a map $A \colon \p \DD^2 \to U(n)$ such that the
trivialization $\nu$ defined as $\nu(\theta, z) = A(z)^{-1}
\nu_0(\theta, z)$ extends to a trivialization of $E$ over any meridian
disk.
(Of course the new trivialization is no longer an orthonormal
trivialization of $F$.)
We extend $\nu$ to a trivialization of $E$ on a regular neighborhood
of $(\SS^1 \times \p \DD^2) \cup (\{\theta_0 \} \times \DD^2)$ for a
fixed $\theta_0 \in \SS^1$.
The complement of this neighborhood in $\SS^2 \times \DD^2$ is a ball.
We can extend $\nu$ inside this ball because $\pi_2(U(n))=0$.
Then $\nu$ defines an isomorphism $(E, F) \cong (\underline{\CC}^n,
F')$, where $F_{(\theta, z)}' = A(z) {\RR^n}$.
The operators~$D^+_\theta$ become $D^+_0 + a_\theta$, where $D^+_0$ is
the standard Cauchy-Riemann operator on $\Gamma(\underline{\CC}^n,
F')$ and $a_\theta \in \Omega^{0,1}(T\DD^2, E)$.
Since $a_\theta$ belongs to a contractible space, the loop $\theta
\mapsto D_\theta^+$ can be continuously deformed to a constant loop.
Then $\dD'$ is a trivial line bundle.
This ends the proof of the orientability of the moduli space
$\widetilde \mMint(\widehat W, S_B ; J)$.
%

%\bibliographystyle{alphabetic}
%\bibliography{main}
\printbibliography

\end{document}